\documentclass[a4paper,10pt]{article}
\usepackage[utf8]{inputenc}
\usepackage{verbatim}
\usepackage[margin=3cm]{geometry}
\usepackage[algo2e,ruled,vlined]{algorithm2e}
\usepackage{algorithm}
\usepackage{here}
\usepackage{xcolor}
\usepackage{setspace}
\usepackage[normalem]{ulem}
\usepackage{algpseudocode}
\usepackage[T1]{fontenc}
\usepackage{dsfont}
\usepackage{comment}
\usepackage{url}
\usepackage[inline]{enumitem}
\usepackage{microtype}
\usepackage[affil-sl]{authblk}
\usepackage{tikz}
\usepackage{amsmath}
\usepackage{amssymb}
\usepackage{amsthm}
\usepackage{diagbox}
\usepackage{mathtools}
\usepackage{nicefrac}
\usepackage{amsfonts,latexsym,graphics}
\usepackage{color}
\usepackage[hidelinks]{hyperref}
\usepackage{array,multirow}
\usepackage[english]{babel}
\usepackage{epsfig}
\usepackage{graphicx}
\usepackage{bm}
\usepackage{makeidx}

\usepackage{xspace}         

\theoremstyle{definition}
\newtheorem{definition}{Definition}[section]
\newtheorem{example}[definition]{Example}

\newtheorem{remark}[definition]{Remark}

\theoremstyle{plain}
\newtheorem{proposition}[definition]{Proposition}
\newtheorem{lemma}[definition]{Lemma}
\newtheorem{theorem}[definition]{Theorem}

\newtheorem{corollary}[definition]{Corollary}
\newcommand*{\claimproofname}{Proof}

\newcommand{\Z}{\mathds{Z}}

\newcommand{\F}{\mathds{F}}
\newcommand{\Q}{\mathds{Q}}

\def\C{{\mbox {\boldmath $C$}}}

\def\matrix0{{\mbox {\boldmath $O$}}}

\def\vec0{\mbox{\bf 0}}

\newcommand\tran{\mkern-2mu\raise1.25ex\hbox{$\scriptscriptstyle\top$}\mkern-3.5mu}

\def\G{\Gamma}

\DeclareMathOperator{\Sym}{Sym}
\DeclareMathOperator{\id}{id}
\DeclareMathOperator{\ord}{ord}
\DeclareMathOperator{\lcm}{lcm}

\def\set#1{\lbrace#1\rbrace}

{\end{list}}

\newcommand{\comments}[1]{}

\title{An infinite class of Neumaier graphs and non-existence results}
\author{Aida Abiad\thanks{\texttt{a.abiad.monge@tue.nl},  Department of Mathematics and Computer Science, Eindhoven University of Technology, The Netherlands\\Department of Mathematics: Analysis, Logic and Discrete Mathematics, Ghent University, Flanders, Belgium\\Department of Mathematics and Data Science of Vrije Universiteit Brussel, Belgium} \quad Wouter Castryck\thanks{\texttt{wouter.castryck@esat.kuleuven.be}, imec-COSIC, Department of Electrical Engineering, KU Leuven, Belgium\\Department of Mathematics: Algebra and Geometry, Ghent University, Flanders, Belgium} \quad Maarten De Boeck\thanks{\texttt{m.de.boeck@tue.nl},  Department of Mathematics and Computer Science, Eindhoven University of Technology, The Netherlands\\Department of Mathematics: Algebra and Geometry, Ghent University, Flanders, Belgium} \quad Jack H. Koolen\thanks{\texttt{koolen@ustc.edu.cn}, School of Mathematical Sciences, University of Science and Technology of China, Hefei, China\\CAS Wu Wen-Tsun Key Laboratory of Mathematics, University of Science and Technology of China, Hefei, China} \quad  Sjanne Zeijlemaker\thanks{\texttt{s.zeijlemaker@tue.nl},  Department of Mathematics and Computer Science, Eindhoven University of Technology, The Netherlands}}

\date{}

\begin{document}
\maketitle

\begin{abstract}
A Neumaier graph is a non-complete edge-regular graph containing a regular clique. A Neumaier graph that is not strongly regular is called a strictly Neumaier graph. In this work we present a new construction of strictly Neumaier graphs, and using Jacobi sums, we show that our construction produces infinitely many instances. Moreover, we prove some necessary conditions for the existence of (strictly) Neumaier graphs that allow us to show that several parameter sets are not admissible.
\end{abstract}

\section{Introduction}

A regular graph is called \emph{edge-regular} if any two adjacent vertices have the same number of common neighbours. A \emph{regular clique} in a regular graph is a clique having the property that every vertex outside of it is adjacent to the same positive number of vertices of the clique, denoted by $e$. A \emph{Neumaier graph} is a non-complete edge-regular graph containing a regular clique. A Neumaier graph that is not a strongly regular graph is called a \emph{strictly Neumaier graph}.
\par In his 1981 paper \cite{neumaier1981regular}, Neumaier studied regular cliques in edge-regular graphs, and he showed that all vertex-transitive, edge-transitive graphs with a regular clique are strongly regular. He subsequently raised the question whether there are edge-regular graphs with a regular clique, that are not strongly regular, i.e.\ whether there are strictly Neumaier graphs.  Greaves and Koolen \cite{greaves2018edge} gave an answer to this question by constructing an infinite family of strictly Neumaier graphs. The same authors provided a second construction in \cite{greaves2019another}. All strictly Neumaier graphs described in \cite{greaves2018edge,greaves2019another} have $e=1$. Evans, Goryainov and Panasenko \cite{evans2018smallest} presented a family of strictly Neumaier graphs which is the only known family with $e>1$. Abiad, De Bruyn,  D'haeseleer and Koolen \cite{dcc2020} investigated Neumaier graphs with few eigenvalues, and showed that Neumaier graphs with four distinct eigenvalues do not exist. 
\par In this article we present a new infinite class of Neumaier graphs, and we also show some non-existence results. In Section \ref{sec:nonexistence} we prove two  new conditions on the parameter set of (strictly) Neumaier graphs (Corollary \ref{coro:firstkilledsetparameters} and Theorem \ref{th:secondkilledsetparameters}), which show that infinitely many parameter sets for strictly Neumaier graphs that had not been ruled out by previous results are not feasible (see  Table \ref{table:parametersstrictlyNG}). In Section \ref{sec:newfamily} we present a new family of (strictly) Neumaier graphs (Theorem \ref{th:constructionpq}). Our construction depends on three parameters: a prime $p$, an odd integer $q$ and an integer $a\in(\Z/pq\Z)^{*}$, fulfilling several conditions. In Section \ref{sec:parameters}, which is purely number-theoretic, we discuss these parameters and show that the family from Section \ref{sec:newfamily} contains an infinite number of strictly Neumaier graphs.

\section{Preliminaries}

Throughout this paper we will consider simple graphs (undirected, loopless, no multiple edges). For a graph $\G$ we denote the set of vertices at distance $i$ from a given vertex $u$ by $\G_{i}(u)$; in particular, the neighbours of $u$ are denoted by $\G_{1}(u)=\G(u)$. Adjacency between vertices is denoted by $\sim$.
\par A graph is \emph{($k$-)regular} if each vertex is adjacent to $k$ vertices. A regular graph is \emph{($\lambda$-)edge-regular} if it is non-empty, and any pair of adjacent vertices has exactly $\lambda$ common neighbours for some integer $\lambda$; it is \emph{($\mu$-)co-edge-regular} if it is not complete, and any pair of non-adjacent vertices has exactly $\mu$ common neighbours for some integer $\mu$. A graph that is both edge-regular and co-edge-regular is called \emph{strongly regular}. An edge-regular graph with parameters $(v,k,\lambda)$ has $v$ vertices, is $k$-regular and $\lambda$-edge-regular; a co-edge-regular graph with parameters $(v,k,\mu)$ has $v$ vertices, is $k$-regular and $\mu$-co-edge-regular. A strongly regular graph has parameters $(v,k,\lambda,\mu)$ if it is edge-regular with parameters $(v,k,\lambda)$ and co-edge-regular with parameters $(v,k,\mu)$.
\par It is immediate that $vk \equiv 0 \pmod{2}$ for a $k$-regular graph with $v$ vertices. We have the following classic result for edge-regular graphs.

\begin{theorem}[{\cite[Section 1.1]{bcn}}]\label{th:ERG}
    Let $\Gamma$ be an edge-regular graph with parameters $(v,k,\lambda)$, then 
    \begin{enumerate}[label=(\roman*)]
        \item $v-2k+\lambda\ge 0$,
        \item $\lambda k \equiv 0 \pmod 2$,
        \item $vk\lambda \equiv 0 \pmod 6$.
    \end{enumerate}
\end{theorem}

Let $\G$ be a graph with vertex set $V(\G)$ and $S\subset V(\G)$. If every vertex in $V(\G) \setminus S$ has precisely $e>0$ neighbors in $S$, we say that $S$ is \emph{$e$-regular}. A \emph{clique} of $\G$ is a subset of $V(\G)$ wherein all vertices are pairwise adjacent; a \emph{coclique} of $\G$ is a subset of $V(\G)$ wherein all vertices are pairwise non-adjacent. 
\par A graph is a \textit{Neumaier graph} with parameters $(v,k,\lambda;e,s)$ if it is edge-regular with parameters $(v,k,\lambda)$ and has an $e$-regular clique of size $s$. A Neumaier graph which is not strongly regular is called \textit{strictly Neumaier}.
\par Neumaier already made the following observations about the regular cliques in Neumaier graphs.
\begin{theorem}[\cite{neumaier1981regular}, Theorem 1.1]\label{th:maxcliques}
     Let $\G$ be a Neumaier graph with parameters $(v,k,\lambda;e,s)$. Then 
     \begin{enumerate}[label=(\roman*)]
        \item the largest clique of $\G$ has size $s$,
        \item all regular cliques are $e$-regular,
        \item the regular cliques are exactly the cliques of size $s$.
    \end{enumerate}
\end{theorem}

Observe that the parameters naturally satisfy $e\le s-1$, $k < v-1$ and $s-2\le \lambda < k$. Theorem \ref{th:existenceconditions} lists some additional conditions on the parameters of Neumaier graphs.

\begin{theorem}[{\cite[Theorem 1.1]{neumaier1981regular} and \cite[Theorem 1]{evans2018smallest}}]\label{th:existenceconditions}
    The parameters $(v,k,\lambda;e,s)$ of a Neumaier graph satisfy the following conditions:
     \begin{enumerate}[label=(\roman*)]
        \item $k-s+e-\lambda -1 \ge 0$,
        \item $s(k-s+1)=(v-s)e$,
        \item $s(s-1)(\lambda-s+2)=(v-s)e(e-1)$.
    \end{enumerate}
\end{theorem}

For strictly Neumaier graphs some additional conditions were derived. We refer to \cite[Proposition 5.1]{greaves2018edge}, \cite[Theorem 1.3]{neumaier1981regular}, \cite[Theorem 4.1]{soicher2015cliques} and \cite[Lemma 4.7]{evansphd}, and  \cite[Theorem 4.10]{evansphd}.

\begin{theorem}\label{th:existenceconditionsstrict}
     The parameters $(v,k,\lambda;e,s)$ of a strictly Neumaier graph satisfy
     \begin{enumerate}[label=(\roman*)]
        \item $s \ge 4$ and, as a result, $\lambda \ge 2$,
        \item $e \le k-2$,
        \item $v \notin \set{2k-\lambda, 2k-\lambda+1}$,
        \item $k-s+e-\lambda -1 \ge 1$.
    \end{enumerate}
\end{theorem}

Table \ref{table:parametersstrictlyNG} lists all parameter sets $(v,k,\lambda;e,s)$ with $v\le 64$ that satisfy the conditions of Theorems \ref{th:ERG}, \ref{th:existenceconditions} and \ref{th:existenceconditionsstrict} (and the trivial conditions mentioned in between), i.e.~the known necessary conditions for the existence of strictly Neumaier graphs. 


\section{Nonexistence results for strictly Neumaier graphs}\label{sec:nonexistence}

\begin{table}[!htb]
\begin{minipage}[t]{.5\linewidth}
\centering
\begin{tabular}{c | c | c | c | c | c }
$v$ & $k$ & $\lambda$ & $e$ & $s$ & Exists?\\\hline\hline
16 & 9 & 4 & 2 & 4 & Yes, \cite{evans2018smallest}\\\hline
21 & 14 & 9 & 4 & 7 & No, Theorem \ref{th:secondkilledsetparameters} \\\hline
22 & 12 & 5 & 2 & 4 &  \\ \hline
24 & 8 & 2 & 1 & 4 & Yes, \cite{evans2018smallest,goryainov2014cayley,greaves2019another} \\  \hline
\multirow[t]{2}{*}{25} & 12 & 5 & 2 & 5 & \\
 & 16 & 9 & 3 & 5 & \\\hline
26 & 15 & 8 & 3 & 6 & \\ \hline
27 & 18 & 12 & 5 & 9 & No, Theorem \ref{th:secondkilledsetparameters}\\ \hline
\multirow[t]{4}{*}{28} & 9 & 2 & 1 & 4 & Yes, \cite{evans2018smallest,greaves2018edge}\\
& \multirow[t]{2}{*}{15} & 6 & 2 & 4 & \\
&  & 8 & 3 & 7 & \\ 
& 18 & 11 & 4 & 7 & \\ \hline
\multirow[t]{2}{*}{33} & 22 & 15 & 6 & 11 & No, Theorem \ref{th:secondkilledsetparameters}\\ 
 & 24 & 17 & 6 & 9 & \\ \hline
34 & 18 & 7 & 2 & 4 & \\ \hline
\multirow[t]{4}{*}{35} & 10 & 3 & 1 & 5 & \\ 
& 16 & 6 & 2 & 5 & \\ 
& 18 & 9 & 3 & 7 & \\ 
& 22 & 12 & 3 & 5 & \\ \hline
\multirow[t]{5}{*}{36} & 11 & 2 & 1 & 4 & \\
& 15 & 6 & 2 & 6 & \\
& 20 & 10 & 3 & 6 & \\ 
& 21 & 12 & 4 & 8 & \\ 
& 25 & 16 & 4 & 6 & \\\hline
\multirow[t]{2}{*}{39} & 26 & 18 & 7 & 13 & No, Theorem \ref{th:secondkilledsetparameters}\\ 
& 30 & 23 & 9 & 13 & No, Corollary \ref{coro:firstkilledsetparameters}\\\hline
\multirow[t]{5}{*}{40} & 12 & 2 & 1 & 4 & Yes, \cite{evans2018smallest}\\ 
& \multirow[t]{2}{*}{21} & 8 & 2 & 4 & \\ 
& & 12 & 4 & 10 & \\ 
& 27 & 18 & 6 & 10 & \\ 
& 30 & 22 & 7 & 10& \\ \hline
\multirow[t]{3}{*}{42} & 11 & 4 & 1 & 6 & \\ 
& 21 & 10 & 3 & 7 & \\ 
& 26 & 15 & 4 & 7 & \\ \hline
44 & 28 & 18 & 6 & 11 & \\\hline
\multirow[t]{8}{*}{45} & 12 & 3 & 1 & 5 & \\
& \multirow[t]{2}{*}{20} & 7 & 2 & 5 & \\
&  & 10 & 3 & 9 & \\
& 24 & 13 & 4 & 9 & \\ 
& \multirow[t]{2}{*}{28} & 15 & 3 & 5 & \\
&  & 17 & 5 & 9 & \\
& 30 & 21 & 8 & 15 & No, Theorem \ref{th:secondkilledsetparameters}\\
& 32 & 22 & 6 & 9 & \\\hline
\multirow[t]{3}{*}{46} & 24 & 9 & 2 & 4 & \\
 & 25 & 12 & 3 & 6 & \\
 & 27 & 16 & 5 & 10 & \\\hline
\multirow[t]{3}{*}{48} & 12 & 4 & 1 & 6 & \\
 & 14 & 2 & 1 & 4 & \\
 & 35 & 26 & 10 & 16 & No, Corollary \ref{coro:firstkilledsetparameters}
\end{tabular}
\end{minipage}%
\begin{minipage}[t]{.5\linewidth}
\centering
\begin{tabular}{c | c | c | c | c | c }
$v$ & $k$ & $\lambda$ & $e$ & $s$ & Exists?\\\hline\hline
\multirow[t]{3}{*}{49} & 18 & 7 & 2 & 7 & \\
 & 24 & 11 & 3 & 7 & \\
 & 30 & 17 & 4 & 7 & \\
 & 36 & 25 & 5 & 7 & \\\hline
50 & 28 & 15 & 4 & 8 & \\\hline
\multirow[t]{2}{*}{51} & 20 & 7 & 2 & 6 & \\
& 34 & 24 & 9 & 17 & No, Theorem \ref{th:secondkilledsetparameters}\\\hline
\multirow[t]{4}{*}{52} & 15 & 2 & 1 & 4 & Yes, \cite{greaves2018edge}\\
&\multirow[t]{2}{*}{27} & 10 & 2 & 4 & \\
& & 16 & 5 & 13 & \\
& 36 & 25 & 8 & 13 & \\\hline
54 & 13 & 4 & 1 & 6 & \\\hline
\multirow[t]{6}{*}{55} & 14 & 3 & 1 & 5 & \\
& 24 & 8 & 2 & 5 & \\
&\multirow[t]{2}{*}{30} & 17 & 5 & 11 & \\
& & 18 & 3 & 5 & \\
& 34 & 21 & 6 & 11 & \\
& 36 & 23 & 6 & 10 & \\\hline
\multirow[t]{4}{*}{56} & 27 & 12 & 3 & 7 & \\
& 30 & 14 & 3 & 6 & \\
& 33 & 20 & 6 & 12 & \\
& 45 & 36 & 12 & 16 & No, Corollary \ref{coro:firstkilledsetparameters}\\\hline
\multirow[t]{4}{*}{57} & 24 & 11 & 3 & 9 & \\
& 38 & 27 & 10 & 19 & No, Theorem \ref{th:secondkilledsetparameters}\\
& 40 & 27 & 6 & 9 & \\
& 42 & 31 & 10 & 15 & \\\hline
58 & 30 & 11 & 2 & 4 & \\\hline
\multirow[t]{4}{*}{60} & 14 & 4 & 1 & 6 & \\
& 17 & 2 & 1 & 4 & \\
& 35 & 22 & 7 & 15 & \\
& 38 & 25 & 8 & 15 & \\\hline
\multirow[t]{8}{*}{63} & 14 & 5 & 1 & 7 & \\
& 30 & 13 & 3 & 7 & \\
& 32 & 16 & 4 & 9 & \\
& \multirow[t]{2}{*}{38} & 21 & 4 & 7 & \\
&& 22 & 5 & 9 & \\
& 42 & 30 & 11 & 21 & No, Theorem \ref{th:secondkilledsetparameters}\\
& 50 & 40 & 15 & 21 & No, Corollary \ref{coro:firstkilledsetparameters}\\
& 52 & 43 & 16 & 21 & No, Corollary \ref{coro:firstkilledsetparameters}\\\hline
\multirow[t]{11}{*}{64} & 18 & 2 & 1 & 4 & \\
& 21 & 8 & 2 & 8 & \\
& 28 & 12 & 3 & 8 & \\
& \multirow[t]{2}{*}{33} & 12 & 2 & 4 & \\
& & 20 & 6 & 16 & \\
& 35 & 18 & 4 & 8 & \\
& 36 & 20 & 5 & 10 & \\
& 42 & 26 & 5 & 8 & \\
& 45 & 32 & 10 & 16 & \\
& 48 & 36 & 11 & 16 & \\
& 49 & 36 & 6 & 8 & 
\end{tabular}
\end{minipage} 
\caption{Feasible parameters for strictly Neumaier graphs up to 64 vertices.}\label{table:parametersstrictlyNG}
\end{table}

In this section we first show a general counting result for co-edge-regular graphs, from which we immediately derive a new condition for Neumaier graphs.

\begin{lemma}\label{lem:coedge}
   If $\Gamma$ is a co-edge-regular graph with parameters $(v,k,\mu)$, then $k(k-1)-\mu(v-k-1)\geq0$. Moreover, if $k(k-1)-\mu(v-k-1)=0$, then $\Gamma$ is strongly regular. If $k(k-1)-\mu(v-k-1)=2$, then each vertex of $\Gamma$ is contained in a unique triangle.
\end{lemma}
\begin{proof} 
Recall that a co-edge-regular graph is not complete. Let $u$ be a vertex in $\Gamma$. Each of the $k$ neighbors of $u$ is adjacent to $k-1$ other vertices. There are $v-k-1$ vertices not adjacent to $u$, which all have exactly $\mu$ common neighbors with $u$. Then there are $\mu(v-k-1)$ edges between a vertex in $\G_{2}(u)$ and a vertex in $\Gamma_{1}(u)$. This number cannot exceed the number of available endpoints in $\Gamma_{1}(u)$, hence $k(k-1) - \mu(v-k-1) \ge 0$.
\par If $k(k-1)-\mu(v-k-1)=0$, then the subgraph induced on $\Gamma_{1}(u)$ is an empty graph, hence any $w\in\Gamma_{1}(u)$ has no common neighbors with $u$. Since $u$ was chosen arbitrarily, $\Gamma$ is strongly regular with parameters $(v,k,0,\mu)$. 
\par Finally, assume that $k(k-1)-\mu(v-k-1)=2$. Then two vertices in $\Gamma_{1}(u)$ are not the endpoints of an edge to $\Gamma_{2}(u)$, which means that the subgraph induced on $\Gamma(u)$ is $(k-2)\cdot K_1\cup K_2$, a graph consisting of a single edge and $k-2$ isolated vertices. Then $u$ is contained in exactly one triangle. As $u$ was arbitrary, this holds for any vertex of $\Gamma$.
\end{proof}

The complement $\overline{\G}$ of a co-egde-regular graph $\G$ is an edge-regular graph, and vice versa. So, if $\Gamma$ is an edge-regular graph with parameters $(v,k,\lambda)$, then $(v-k-1)(v-k-2)-k(v-2k+\lambda)\geq0$. In particular, observe that from Lemma \ref{lem:coedge} it follows that if $k(k-1)-\mu(v-k-1)=0$, then not only $\Gamma$ is strongly regular, but also $\overline{\Gamma}$ is strongly regular; the latter has parameters $(v,v-k-1,v-2-2k+\mu, v-2k)$.
\par Looking at the complement of an edge-regular graph, we can deduce the following result.

\begin{corollary}\label{coro:firstkilledsetparameters}
    There are no edge-regular graphs (and hence no Neumaier graphs) with parameter set $(v,k,\lambda)$ such that $(v-k-1)(v-k-2)-k(v-2k+\lambda)<0$. All edge-regular graphs (and thus also all Neumaier graphs) with parameter set $(v,k,\lambda)$ such that $(v-k-1)(v-k-2)-k(v-2k+\lambda)=0$ are strongly regular.
\end{corollary}

Corollary \ref{coro:firstkilledsetparameters} allows to reduce the number of admissible parameter sets. Actually, it also follows from the proof of Lemma \ref{lem:coedge} that $k(k-1)-\mu(v-k-1)$ is even, but this is not useful further on to reduce the number of admissible parameter sets since we already know that $vk$ and $k\lambda$ are both even for edge-regular graphs with parameters $(v,k,\lambda)$.

\begin{remark}
    It follows from Corollary \ref{coro:firstkilledsetparameters} that several parameter sets that were admissible as parameter sets of strictly Neumaier graphs by Theorems \ref{th:ERG}, \ref{th:existenceconditions} and \ref{th:existenceconditionsstrict} are now showed not to be admissible as such. In particular, there are 14 parameter sets with $v\leq 100$ that are now showed not to be parameter sets of Neumaier graphs: twelve of them have $(v-k-1)(v-k-2)-k(v-2k+\lambda)<0$, and $(56, 45, 36, 12, 16)$ and $(77, 60, 47, 15, 21)$ can only correspond to strongly regular Neumaier graphs. Note that the former parameter set corresponds to the complement of a $(56,10,0,2)$ strongly regular graph, and the latter to the complement of a $(77,16,0,4)$ strongly regular graph. The Sims-Gewirtz graph and the Mesner-M22 graph are the unique strongly regular graphs with these parameters, respectively, see \cite{BMbook}. The complements of the Sims-Gewirtz and the Mesner-M22 graph admit a 12-regular clique of size 16, and a 15-regular clique of size 21, respectively, so are indeed Neumaier.
    \par We show the strength of Corollary \ref{coro:firstkilledsetparameters} by giving several infinite families of parameter sets that are admissible by Theorems \ref{th:ERG}, \ref{th:existenceconditions} and \ref{th:existenceconditionsstrict}, but which do not meet the conditions of Corollary \ref{coro:firstkilledsetparameters}. The parameter sets
    \[
        \left(3\frac{a^{n+1}-1}{a-1},2a^n+a\frac{a^{n}-1}{a-1},3a^{n}-2a^{n-1}+a\frac{a^{n-1}-1}{a-1}-1;a^{n},\frac{a^{n+1}-1}{a-1}\right)
    \]
    with integers $a,n\geq2$, fulfill all conditions of Theorems \ref{th:ERG}, \ref{th:existenceconditions} and \ref{th:existenceconditionsstrict}, but
    \[
        (v-k-1)(v-k-2)-k(v-2k+\lambda)=-2\frac{(a-2)a^{n}(a^{n-1}-2)-1}{a-1}\,
    \]
    is negative if $a\geq3$. So, in case $a\geq3$ there are no Neumaier graphs with these parameters by Corollary \ref{coro:firstkilledsetparameters}. In case $a=2$, then all Neumaier graphs with these parameters are strongly regular. Likewise, the parameter sets $(27a+21,21a+14,13a+7;6a+4,9a+7)$, with $a\geq0$ an integer, fulfill the conditions of Theorems \ref{th:ERG}, \ref{th:existenceconditions} and \ref{th:existenceconditionsstrict}, but
    \[
        (v-k-1)(v-k-2)-k(v-2k+\lambda)=-2(a+1)(3a-1)\,
    \]
    is negative if $a\geq1$. So, in case $a\geq1$ there are no Neumaier graphs with these parameters by Corollary \ref{coro:firstkilledsetparameters}.
    \par For the parameter sets
    \begin{align*}
        &(a^{2}(2a+3),(a+1)(4a^{2}-1),4a^3+2a^2+a-2;4a^2-2a,4a^2)\quad\text{ and}\\
        &(2(2a+1)(a^2+a-1),2(a+1)(2a^{2}-1),4a^3+2a^2+a-3;4a^2-2a,4a^2-1)
    \end{align*}
    with $a\geq2$ an integer, all conditions from Theorems \ref{th:ERG}, \ref{th:existenceconditions} and \ref{th:existenceconditionsstrict} are fulfilled, but we have $(v-k-1)(v-k-2)-k(v-2k+\lambda)=0$. So any Neumaier graph with these parameters must be strongly regular.
\end{remark}

The next result shows the nonexistence of certain strictly Neumaier graphs with $(v-k-1)(v-k-2)-k(v-2k+\lambda)=2$. Note again that this parameter set fulfills all conditions from Theorems \ref{th:ERG}, \ref{th:existenceconditions} and \ref{th:existenceconditionsstrict}

{
\begin{theorem}\label{th:secondkilledsetparameters}
    There is no Neumaier graph with parameter set $(6l+3,4l+2,3l;l+1,2l+1)$ for any integer $l\geq 3$.
\end{theorem}
\begin{proof}
Suppose that $\Gamma$ is a Neumaier graph with parameters $(6l+3,4l+2,3l;l+1,2l+1)$ for some integer $l\geq 3$. Its complement $\overline{\Gamma}$ is a co-edge-regular graph with parameters $(6l+3,2l,l-1)$. By Lemma \ref{lem:coedge} we know that each vertex of $\overline{\Gamma}$ is in a unique triangle.
\par We also know that $\overline{\Gamma}$ has an $l$-regular coclique $C$ of order $2l+1$, arising from an $(l+1)$-regular clique in $\Gamma$. Let $C = \set{x_1,\dots,x_{2l+1}}$ and let $\set{x_i,y_i,z_i}$ denote the triangle containing $x_i$. Without loss of generality, $z_1, y_2, \dots, y_{l}$ are the neighbors of $y_1$ that are not in $C$; here we used that $y_{1}$ has at most one neighbor in each triangle. Note that $y_i$ and $y_j$ cannot be neighbors for any $i,j \in \{2,\dots,l\}$, as $y_1$ is in only one triangle, namely $\set{x_1,y_1,z_1}$. Furthermore, we can assume that $x_1,x_{l+1},\dots x_{2l-1}$ are the neighbors of $y_1$ in $C$ (observe that $y_1\not\sim x_i$ for $i \in \set{2,\dots,l}$, because this would create a triangle $\set{x_i,y_i,y_1}$). Then,  for any $j\in\{2,\dots,l\}$, the vertex $y_j$ is not adjacent to any $x_i \in \{x_1\}\cup\{x_{l+1},\dots,x_{2l-1}\}$, since this would induce a triangle $\set{y_{j}, x_i, y_1}$. We know that $l\geq3$. Now, by the $l$-regularity of $C$, $y_2$ and $y_3$ each have $l$ neighbors in $\set{x_2, \dots, x_{l}}\cup \set{x_{2l},x_{2l+1}}$. This means that they have at least $l-1$ common neighbors in this set, contradicting the $(l-1)$-co-edge-regularity of $\overline{\Gamma}$, since $y_1$ is also a common neighbor of $y_2$ and $y_3$.
\end{proof}
}
As a consequence of Theorem \ref{th:secondkilledsetparameters}, we can settle down several open cases of existence of strictly Neumaier graphs, see~\cite[Table 2]{evans2018smallest}. The updated list of feasible parameters for strictly Neumaier graphs up to 64 vertices is shown in Table \ref{table:parametersstrictlyNG}.

\section{A new family of strictly Neumaier graphs}\label{sec:newfamily}

In \cite{greaves2019another} Greaves and Koolen described a construction of strictly Neumaier graphs arising from antipodal distance-regular graphs with diameter 3. It was later generalised by Evans in his PhD thesis, see \cite[Theorem 5.1]{evansphd}; this generalisation also appeared in \cite{egkm}. Next we will describe the construction from \cite{evansphd}, for later use. A \emph{spread} of (the vertex set of) a graph is a partition of the vertex set in subsets, i.e.~a family of pairwise disjoint subsets of the vertex set whose union is the whole vertex set. 

\begin{definition}
    Let $\Gamma_{1}=(V_{1},E_{1}),\dots,\Gamma_{t}=(V_{t},E_{t})$ be $t$ graphs such that for any $i=1,\dots,t$ the graph $\Gamma_{i}$ admits a spread of 1-regular cocliques, denoted by $C_{i,1},\dots,C_{i,a}$. Let $\pi_{1}=\id,\pi_{2},\dots,\pi_{t}$ be $t$ permutations in $\Sym_{a}$. The graph $F_{(\pi_{2},\dots,\pi_{t})}(\Gamma_{1},\dots,\Gamma_{t})$ is the graph that has vertex set $V_{1}\cup\dots\cup V_{t}$ and where two vertices $x\in C_{i,k}$ and $y\in C_{j,l}$ are adjacent if and only if $i=j$ and $x\sim y$ in $\Gamma_{i}$, or if $\pi^{-1}_{i}(k)=\pi^{-1}_{j}(l)$. In particular, $\Gamma_{1},\dots,\Gamma_{t}$ could be $t$ copies of the same edge-regular graph $\Gamma$. In this case we denote $F_{(\pi_{2},\dots,\pi_{t})}(\Gamma_{1},\dots,\Gamma_{t})$ by $F_{(\pi_{2},\dots,\pi_{t})}(\Gamma)$.
\end{definition}

In other words, in the previous construction we take the graphs $\Gamma_{1},\dots,\Gamma_{t}$ and for any $k$ we add the edges between all vertices in $C_{1,k}\cup C_{2,\pi_{2}(k)}\cup\dots\cup C_{t,\pi_{t}(k)}$. In \cite[Theorem 5.1]{evansphd} the author describes the 1-regular cocliques as perfect 1-codes, but they are just equivalent. Also, in the above construction we actually could do without the permutations $\pi_{1}=\id,\pi_{2},\dots,\pi_{t}$, as we could change the order on the cocliques in each of the graphs. We do however want to point out that we can obtain several not necessarily isomorphic (actually almost always non-isomorphic) graphs starting from the same set of edge-regular graphs.
\par The following result is essential to the rest of the paper.


\begin{theorem}[{\cite[Theorem 5.1]{evansphd}}]\label{th:mainconstruction}
    Let $\Gamma_{1}=(V_{1},E_{1}),\dots,\Gamma_
    {t}=(V_{t},E_{t})$ be $t$ edge-regular graphs with parameters $(v,k,\lambda)$ such that for any $i=1,\dots,t$ the graph $\Gamma_{i}$ admits a spread of 1-regular cocliques, $C_{i,1},\dots,C_{i,k+1}$. Let $\pi_{1}=\id,\pi_{2},\dots,\pi_{t}$ be $t$ permutations in $\Sym_{k+1}$. If $t=\frac{(\lambda+2)(k+1)}{v}$, then $F_{(\pi_{2},\dots,\pi_{t})}(\Gamma_{1},\dots,\Gamma_{t})$ is a Neumaier graph with parameters $(vt,k+\lambda+1,\lambda;1,\lambda+2)$, which admits a spread of $1$-regular cliques.
\end{theorem}

\begin{remark}
    Note that in the construction from Theorem \ref{th:mainconstruction} the number of cocliques is precisely one more than the regularity parameter $k$, since each vertex has precisely one neighbour in each of the cocliques of the spread, and no neighbour in its own coclique. This was not pointed out in \cite[Theorem 5.1]{evansphd}, where the regularity and the number of cocliques were two independent parameters.
\end{remark}

\begin{remark}
    The construction from Theorem \ref{th:mainconstruction} always produces a Neumaier graph with $e=1$ since it requires a spread of 1-regular cocliques in each of the graphs. There is no straightforward generalisation of this construction for $e>1$, starting from $e$-regular cocliques, since two (adjacent) vertices in $C_{i,k}$ would not have the same number of common neighbours as two (adjacent) vertices, one in $C_{i,k}$ and one in $C_{j,k}$, $i\neq j$, violating the edge-regularity. Here we used the notation from Theorem \ref{th:mainconstruction}.
\end{remark}

The next theorem gives checks when the construction from Theorem \ref{th:mainconstruction} produces strictly Neumaier graphs. The first case was recently also described in \cite[Theorem 1]{egkm}, independently from this paper.

\begin{theorem}\label{th:mainconstructionstrictly}
    Let $\Gamma_{1}=(V_{1},E_{1}),\dots,\Gamma_
    {t}=(V_{t},E_{t})$ be $t$ edge-regular graphs with parameters $(v,k,\lambda)$ such that $vt=(\lambda+2)(k+1)$ and such that for any $i=1,\dots,t$ the graph $\Gamma_{i}$ admits a spread of 1-regular cocliques, $C_{i,1},\dots,C_{i,k+1}$. Let $\pi_{1}=\id,\pi_{2},\dots,\pi_{t}$ be $t$ permutations in $\Sym_{k+1}$. If
    \begin{itemize}
        \item $t\geq2$ and the $\Gamma_{i}$'s are not complete, or
        \item $t=1$ and there are two vertices in $\Gamma_{1}$ that are at distance at least 3 and not in the same $C_{1,j}$, 
    \end{itemize}
    then the graph $F_{(\pi_{2},\dots,\pi_{t})}(\Gamma_{1},\dots,\Gamma_{t})$ is a strictly Neumaier graph.
\end{theorem}
\begin{proof}
    We denote $F_{(\pi_{2},\dots,\pi_{t})}(\Gamma_{1},\dots,\Gamma_{t})$ by $\Gamma$. Note that if $\Gamma_{1}$ is complete, then $v=k+1=\lambda+2$, hence $t=v\geq 2$. So, in each of the two cases above we know that $\Gamma_{1}$ is not complete.
    \par Let $u_{1},w_{1}\in V_{1}$ be two vertices that are at distance two in $\Gamma_{1}$; these exist since $\Gamma_{1}$ is not complete. Then there is a vertex $x$ such that $u_{1}\sim x\sim w_{1}$. Since $x$ cannot have two neighbours in the same coclique of $\Gamma_{1}$, we find that $u_{1}$ and $w_{1}$ belong to different cocliques, say $u_{1}\in C_{1,1}$ and $w_{1}\in C_{1,2}$. In $\Gamma_{1}$ there is precisely one vertex $w'_{1}\in \Gamma_{1}(u_{1})\cap C_{1,2}$ and precisely precisely one vertex $u'_{1}\in \Gamma_{1}(w_{1})\cap C_{1,1}$ by the 1-regularity of the cocliques. Obviously $u'_{1}\neq x\neq w'_{1}$. So, in $\Gamma$ the vertices $u_{1},w_{1}$ have at least three common neighbours.
    \par If $t\geq2$, we can find a vertex $w_{2}\in C_{2,2}$. From the construction it follows immediately that in $\Gamma$ the vertices $u_{1}$ and $w_{2}$ have precisely two common neighbours, one in $C_{1,2}$ and one in $C_{2,1}$. So as $\{u_{1},w_{1}\}$ and $\{u_{1},w_{2}\}$ have a different number of common neighbours, $\Gamma$ is not co-edge-regular, so not strongly regular, and thus a strictly Neumaier graph.
    \par Now we consider the case with $t=1$. 
    Assume now there are vertices $y$ and $y'$ in $\Gamma_{1}$ such that $d(y,y')\geq3$ and $y$ and $y'$ are in different cocliques of $\Gamma_{1}$, say $y\in C_{1,m}$ and $y'\in C_{1,m'}$. In $\Gamma$ the vertex $y$ has a unique neighbour $z\in C_{1,m'}$, and $y'$ has a unique neighbour $z'\in C_{1,m}$. So, the vertices $z$ and $z'$ are common neighbours of $y$ and $y'$ in $\Gamma$. Any other common neighbour of $y$ and $y'$ in $\Gamma$ cannot be in $C_{1,m}\cup C_{1,m'}$ by construction, so must be a common neighbour of $y$ and $y'$ in $\Gamma_{1}$. But such a vertex cannot exist since $d(y,y')\geq3$. It follows that $y$ and $y'$ have precisely two common neighbours in $\Gamma$. But we know from the beginning of the proof that there are two vertices in $\Gamma$ that have precisely three common neighbours. So, the graph $\Gamma$ cannot be strongly regular, so is a strictly Neumaier graph.
\end{proof}

Given Theorems \ref{th:mainconstruction} and \ref{th:mainconstructionstrictly} it is essential to find (families of) edge-regular graphs with a spread of $1$-regular cocliques. Essentially all known constructions of strictly Neumaier graphs with $e=1$ arise from this construction. In \cite{greaves2019another} the authors use $a$-antipodal distance-regular graphs of diameter 3; examples of these include the Taylor graphs, the Thas-Somma graphs, and the graphs constructed by Brouwer, Hensel and Mathon.
\par In \cite{evansphd} Evans describes some particular applications of this Theorem \ref{th:mainconstruction}, including the construction of a strictly Neumaier graph on 40 vertices and one on 78 vertices. In \cite{greaves2018edge} Greaves and Koolen constructed a family of strictly Neumaier graphs as Cayley graphs on the group $\Z/l\Z\times(\Z/2\Z)^{m}\times(\F_{q},+)$, with $m\in\{2,3\}$. It can however be seen that the restricted Cayley graph on $(\Z/2\Z)^{m}\times(\F_{q},+)$ produces an edge-regular graph that admits a spread of 1-regular cocliques, and that the graphs described in \cite{greaves2018edge} appear through an application of Theorem \ref{th:mainconstruction} (the factor $\Z/{l}\Z$ produces $l$ copies of this graph, all with the same ordering on the cocliques).

\medskip\par We will now describe a new construction of edge-regular graphs having a spread of 1-regular cocliques. 

\begin{definition}
    Let $n$ be an integer and $a\in(\Z/n\Z)^{*}$ such that $a^{i}\equiv-1\pmod{n}$, where $2i$ is the order of $a$ in $(\Z/n\Z)^{*},\cdot$. Then $S_{n}(a)$ is the set $\{a^{j}\in\Z/n\Z\mid 0\leq j<2i\}$ and $\Gamma_{n}(a)$ is the Cayley graph on $\Z/n\Z,+$ with $S_{n}(a)$ as generating set.
\end{definition}

\begin{theorem}\label{th:twoprimes}
    Let $p$ be an odd prime and let $q$ be an odd integer. If $a\in(\Z/pq\Z)^*$ is such that $a\pmod{p}$ is a generator of $(\Z/p\Z)^*,\cdot$ and such that $a^{\frac{p-1}{2}}\equiv-1\pmod{pq}$, then the Cayley graph $\Gamma_{pq}(a)$ is an edge-regular graph with parameters $(pq,p-1,\lambda)$, with $\lambda=\left|S_{pq}(a)\cap(S_{pq}(a)+1)\right|$, that has a spread of 1-regular cocliques.
\end{theorem}
\begin{proof}
    We denote $S_{pq}(a)$ by $S$ and $\Gamma_{pq}(a)$ by $\Gamma$. First note that $S=-S$ since $-1\in S$ and that $|S|=p-1$. Obviously $\Gamma$ is $(p-1)$-regular. Since $\Gamma$ is a Cayley graph and thus  vertex-transitive, it is sufficient to check that $|\Gamma(0)\cap\Gamma(a^{i})|=|S\cap(S+1)|$ for all $i=0,\dots,p-2$. Now,
    \begin{align*}
        |\Gamma(0)\cap\Gamma(a^{i})|&=|\{a^{j}\mid a^{j}-a^{i}\in S\}|\\
        &=|\{a^{j}\mid \exists k: a^{j}-a^{i}=a^{k}\}|\\
        &=|\{a^{j}\mid \exists k: a^{j-i}=a^{k-i}+1\}|\\
        &=|\{a^{j'}\mid \exists k': a^{j'}=a^{k'}+1\}|\\
        &=|\{a^{j'}\mid \exists s\in S: a^{j'}=s+1\}|\\
        &=|S\cap(S+1)|\\
        &=\lambda\;,
    \end{align*}
    which shows that $\Gamma$ is edge-regular with parameters $(pq,p-1,\lambda)$.
    \par Let $H$ be the subgroup of $\Z/(pq\Z),+$ generated by the integer $p$; this subgroup has order $q$. It is clear that $S\cap H=\emptyset$. Moreover, a coset of $H$ contains at most one element of $S$ since $p\mid a^{i}-a^{j}$ implies that $a^{i-j}=1\pmod{p}$. Since $|S|=p-1$, each coset of $H$ contains precisely one element of $S\cup\{0\}$. In other words, each element of $\Z/(pq\Z)$ can be written in a unique way as the sum of an element in $H$ and an element in $S\cup\{0\}$. Consequently, each coset of $H$, including $H$ itself is a 1-regular coclique of $\Gamma$. Clearly, the cosets of $H$ form a spread.
\end{proof}
 
\begin{remark}
    In the proof of the previous theorem it is clear that the 1-regular cocliques correspond to the cosets of a subgroup of $\Z/pq\Z,+$. Cayley graphs on a group $G$ wherein a 1-regular coclique corresponds to a subgroup of the group $G$ are called \emph{subgroup perfect codes}. These are interesting in their own right. We refer to \cite{hxz} for a brief survey and to \cite{zhangzhou} for recent work on this topic.
\end{remark}

Using Theorem~\ref{th:mainconstruction} and Theorem~\ref{th:twoprimes}, we can now state our main result of this section.

\begin{theorem}\label{th:constructionpq}
    Let $p$ and $q$ be two different odd primes and let $a\in(\Z/pq\Z)^*$ be a generator of $(\Z/p\Z)^*$ and such that $a^{\frac{p-1}{2}}=-1\pmod{pq}$. Write $S=S_{pq}(a)$. If $|S\cap(S+1)|\equiv-2\pmod{q}$, then $F_{(\pi_{2},\dots,\pi_{t})}(\Gamma_{pq}(a))$, with $t=\frac{|S\cap(S+1)|+2}{q}$ and $\pi_i\in\Sym_{p}$ for $i=2,\dots,t$, is a Neumaier graph with parameters $(tpq,p+|S\cap(S+1)|,|S\cap(S+1)|;1,|S\cap(S+1)|+2)$.
\end{theorem}
\begin{proof}
    It follows from Theorem \ref{th:twoprimes} that $\Gamma_{pq}(a)$ is an edge-regular graph with a spread of 1-regular cocliques. The theorem then follows from an application of Theorem \ref{th:mainconstruction}.
\end{proof}

\begin{remark}
    Tables \ref{tab:constructionqprime} and \ref{tab:constructionqnonprime} contain several parameter sets $(q,p,a)$ for which indeed $|S_{pq}(a)\cap(S_{pq}(a)+1)|\equiv-2\pmod{q}$ and thus a Neumaier graph can be constructed using Theorem \ref{th:constructionpq}. Note that in general many non-isomorphic examples can be constructed by chosing different $\pi_i\in\Sym_{p}$, for $i=2,\dots,t$, if $t\geq2$.  
    If $\gcd(i,p-1)=1$, then $a$ and $a^{i}$ clearly generate the same subgroup of $(\Z/pq\Z)^*$, so $\Gamma_{pq}(a)$ and $\Gamma_{pq}(a^{i})$ are equal. So, in Tables \ref{tab:constructionqprime} and \ref{tab:constructionqnonprime} only one generator for each subgroup is given.
    \par We also point out that if $|S_{pq}(a)\cap(S_{pq}(a)+1)|\equiv-2\pmod{q}$, then $|S_{pq}(a)\cap(S_{pq}(a)+1)|\geq q-2$. It follows immediately that $p>q$. In particular $\gcd(p,q)=1$.
\end{remark}

\begin{remark}
    In most applications of Theorem \ref{th:constructionpq} we have $t\geq2$. We know by Theorem \ref{th:mainconstructionstrictly} that in these cases the construction produces strictly Neumaier graphs. However, also when $t=1$, the construction in Theorem \ref{th:constructionpq} often produces a strictly Neumaier graph, e.g. the graph $F(\Gamma_{65}(2))$ is a strictly Neumaier graph. This is the smallest graph that arises from this construction.
\end{remark}

\begin{remark}
    In Theorem \ref{th:constructionpq} we take $t$ copies of the graph $\Gamma_{pq}(a)$. However, there are also other options in some cases. E.g. if $q=13$ and $p=397$, both $\Gamma_{5161}(6)$ and $\Gamma_{5161}(20)$ are edge-regular graphs with parameters (5161,396,24), but these graphs are not isomorphic (it can be checked that they have different spectrum). We know that $F_{\pi_{2}}(\Gamma_{5161}(6))$ and $F_{\pi_{2}}(\Gamma_{5161}(20))$ are strictly Neumaier graphs for any $\pi_{2}\in\Sym_{397}$, but we can also apply Theorem \ref{th:mainconstruction} with one copy of each: $F_{\pi_{2}}(\Gamma_{5161}(6),\Gamma_{5161}(20))$ is also a strictly Neumaier graph for any $\pi_{2}\in\Sym_{397}$.
\end{remark}

\begin{table}[!htp]
    \centering
    \begin{tabular}{c|c|c|c||c|c|c|c}
        $q$ & $p$ & $a$ & $t$ & $v$ & $k$ & $\lambda$ & $s$ \\\hline\hline
        5 & 13 & 2 & 1 & 65 & 16 & 3 & 5\\\cline{2-8}
          & 37 & 2 & 1 & 185 & 40 & 3 & 5\\\cline{2-8}
          & 61 & 17 & 4 & 1220 & 79 & 18 & 20\\\cline{2-8}
          & 149 & 13 & 4 & 2980 & 167 & 18 & 20\\\cline{3-8}
          & & 2 & 7 & 5215 & 182 & 33 & 35\\\cline{2-8}
          & 197 & 3 & 10 & 9850 & 245 & 48 & 50\\\cline{2-8}
          & 269 & 3 & 10 & 13450 & 317 & 48 & 50\\\cline{3-8}
          &  & 2 & 13 & 17485 & 332 & 63 & 65\\\cline{2-8}
          & 293 & 2 & 13 & 19045 & 356 & 63 & 65\\\cline{2-8}
          & 397 & 13 & 13 & 25805 & 460 & 63 & 65\\\cline{2-8}
          & 421 & 2 & 13 & 27365 & 484 & 63 & 65\\\cline{2-8}
          & 557 & 13 & 22 & 61270 & 665 & 108 & 110\\\cline{2-8}
          & 613 & 13 & 22 & 67430 & 721 & 108 & 110\\\cline{2-8}
          & 661 & 18 & 28 & 92540 & 799 & 138 & 140\\\cline{2-8}
          & 677 & 7 & 22 & 74470 & 785 & 108 & 110\\\cline{2-8}
          & 701 & 2 & 31 & 108655 & 854 & 153 & 155\\\cline{2-8}
          & 773 & 3 & 34 & 131410 & 941 & 168 & 170\\\cline{2-8}
          & 821 & 2 & 31 & 127255 & 974 & 153 & 155\\\cline{2-8}
          & 829 & 47 & 28 & 116060 & 967 & 138 & 140\\\cline{3-8}
          & & 2 & 31 & 128495 & 982 & 153 & 155\\\cline{2-8}
          & 853 & 18 & 28 & 119420 & 991 & 138 & 140\\\hline
        7 & 79 & 54 & 1 & 553 & 84 & 5 & 7\\\cline{2-8}
          & 103 & 45 & 1 & 721 & 108 & 5 & 7\\\cline{2-8}
          & 127 & 12 & 2 & 1778 & 139 & 12 & 14\\\cline{2-8}
          & 139 & 26 & 4 & 3892 & 165 & 26 & 28\\\cline{2-8}
          & 307 & 45 & 8 & 17192 & 361 & 54 & 56\\\cline{2-8}
          & 379 & 10 & 8 & 21224 & 433 & 54 & 56\\\cline{2-8}
          & 487 & 3 & 8 & 27272 & 541 & 54 & 56\\\cline{2-8}
          & 547 & 33 & 16 & 61264 & 657 & 110 & 112\\\cline{2-8}
          & 571 & 3 & 16 & 63952 & 681 & 110 & 112\\\cline{2-8}
          & 631 & 3 & 11 & 48587 & 706 & 75 & 77\\\cline{2-8}
          & 691 & 12 & 16 & 77392 & 801 & 110 & 112\\\hline
        11 & 131 & 2 & 1 & 1441 & 140 & 9 & 11\\\cline{2-8}
           & 991 & 6 & 10 & 109010 & 1099 & 108 & 110\\\hline
        13 & 61 & 2 & 1 & 793 & 72 & 11 & 13\\\cline{2-8}
           & 397 & 6 & 2 & 10322 & 421 & 24 & 26\\\cline{3-8}
           & & 20 & 2 & 10322 & 421 & 24 & 26\\\cline{2-8}
           & 829 & 2 & 5 & 53885 & 892 & 63 & 65\\\hline
        17 & 977 & 23 & 1 & 16609 & 992 & 15 & 17\\\hline
    \end{tabular}
    \caption{Parameter sets $(q,p,a)$, with $q\leq 17$ and $p\leq 1000$, for which the conditions in Theorem \ref{th:constructionpq} are fulfilled. We give the parameter $t$ and the parameters of the resulting Neumaier graphs. Recall that $e=1$.}
    \label{tab:constructionqprime}
\end{table}

\begin{table}[!htp]
    \centering
    \begin{tabular}{c|c|c|c||c|c|c|c}
        $q$ & $p$ & $a$ & $t$ & $v$ & $k$ & $\lambda$ & $s$ \\\hline\hline
        25 & 1021 & 77 & 2 & 51050 & 1069 & 48 & 50\\\cline{3-8}
           & & 122 & 2 & 51050 & 1069 & 48 & 50\\\cline{2-8}
           & 1181 & 42 & 2 & 59050 & 1229 & 48 & 50\\\cline{2-8}
           & 1301 & 3 & 2 & 65050 & 1349 & 48 & 50\\\cline{3-8}
           & & 73 & 2 & 65050 & 1349 & 48 & 50\\\cline{2-8}
           & 1381 & 42 & 2 & 69050 & 1429 & 48 & 50\\\cline{3-8}
           & & 123 & 2 & 69050 & 1429 & 48 & 50\\\cline{2-8}
           & 1621 & 88 & 2 & 81050 & 1669 & 48 & 50\\\cline{3-8}
           & & 113 & 2 & 81050 & 1669 & 48 & 50\\\cline{2-8}
           & 1741 & 197 & 2 & 87050 & 1789 & 48 & 50\\\cline{2-8}
           & 2141 & 58 & 2 & 107050 & 2189 & 48 & 50\\\cline{3-8}
           &  & 112 & 2 & 107050 & 2189 & 48 & 50\\\hline   
    \end{tabular}
    \caption{Parameter sets $(q,p,a)$, with $q=25$ and $p\leq 2400$, for which the conditions in Theorem \ref{th:constructionpq} are fulfilled. We give the parameter $t$ and the parameters of the resulting Neumaier graphs. Recall that $e=1$.}
    \label{tab:constructionqnonprime}
\end{table}

\section{Discussion of the parameters}\label{sec:parameters}

Given the construction of (strictly) Neumaier graphs in Theorem \ref{th:constructionpq}, we wonder for which odd integers $q$ we can find primes $p$ and corresponding integers $a$ satisfying the stated conditions. We know from Tables \ref{tab:constructionqprime} and \ref{tab:constructionqnonprime} that there are indeed such parameter sets $(q,p,a)$. In particular we ask ourselves whether the construction from Theorem \ref{th:constructionpq} produces an infinite number of (strictly) Neumaier graphs, and whether for any $q$ we can find a prime $p$ and an integer $a$ satisfying the conditions.

Regarding the first question, we will show that actually there is an infinite number of odd integers $q$ such that for each of them there is an infinite number of primes $p$ for which an integer $a$ exists, satisfying the conditions from Theorem \ref{th:constructionpq}, thereby showing that the construction from this theorem produces an infinite number of (strictly) Neumaier graphs. We refer to Sections \ref{ssec:infinite5}, \ref{ssec:infinite7} and \ref{ssec:infiniteinf}. For $q=5$ and $q=7$ we also determine the density of the primes $p$ for which an admissible $a$ exists. The proofs in these sections rely on a formula given in Section \ref{ssec:formula}, which involves Jacobi sums. Therefore we give a gentle introduction to Jacobi sums in Section \ref{ssec:jacobi}. \par We investigate the second question in Section \ref{ssec:nonadmissible}, obtaining some values of $q$ that are not admissible.


\subsection{Non-admissible \texorpdfstring{$q$}{q}'s}\label{ssec:nonadmissible}

Note that $q=3$ and $q=9$ are notably absent from Tables \ref{tab:constructionqprime} and \ref{tab:constructionqnonprime}. We will show that this is no coincidence. In Remark \ref{rem:nomultiplesof3} we will see that $q$ cannot be a multiple of 3.

\begin{theorem}\label{th:twoprimesmodulo}
	Let $p$ be an odd prime, let $q$ be an odd integer and let $a\in(\Z/pq\Z)^*$ be such that $a\pmod{p}$ is a generator of $(\Z/p\Z)^*,\cdot$ and such that $a^{\frac{p-1}{2}}\equiv-1\pmod{pq}$. Denote the set of elements of order $6$ in $S_{pq}(a)$ by $Z_{6}$ (if there are none $Z_{6}=\emptyset$). Then $\left|S_{pq}(a)\cap(S_{pq}(a)+1)\right|\equiv3\delta+2\epsilon\pmod{6}$, where
	\begin{align*}
		\delta=
		\begin{cases}
			1 & \text{if } 2\in S_{pq}(a)\\
			0 & \text{if } 2\notin S_{pq}(a)
		\end{cases}
		&&
		\epsilon=
		\begin{cases}
			1 & \text{if } Z_{6}\cap (S_{pq}(a)+1)\neq\emptyset\\
			0 & \text{if } Z_{6}\cap (S_{pq}(a)+1)=\emptyset
		\end{cases}\;.
	\end{align*} 
\end{theorem}
\begin{proof}
	We denote $S_{pq}(a)$ by $S$. Define the maps $\varphi$ on $\Z/pq\Z$ and $\psi$ on $(\Z/pq\Z)^{*}$ as follows: $\varphi(x)=1-x$ and $\psi(x)=\frac{x-1}{x}$. If $b\in S\cap(S+1)$, then there are integers $m,n$ such that $b=a^{m}=a^{n}+1$, and we can see that
	\begin{align*}
		\varphi(b)&=1-(a^{n}+1)=a^{\frac{p-1}{2}+n} &\psi(b)&=\frac{(a^{n}+1)-1}{a^{m}}=a^{n-m}\\
		&=1-a^{m}=a^{\frac{p-1}{2}+m}+1 &&=\frac{a^{m}-1}{a^{m}}=1+a^{\frac{p-1}{2}-m}\;,
	\end{align*}
	hence $\varphi(b),\psi(b)\in S\cap (S+1)$. So, we can look at the restriction of $\varphi$ and $\psi$ to $S\cap (S+1)$; note that $S\subseteq(\Z/pq\Z)^{*}$, and that $1\notin S+1$. We will denote these restrictions also by $\varphi$ and $\psi$. It can easily be seen that $\varphi^{2}=id=\psi^{3}$ and that $\varphi\circ\psi=\psi^{2}\circ\varphi$. So the group $G=\left\langle\varphi,\psi\right\rangle$ is isomorphic to $S_{3}$ and acts naturally on $S\cap (S+1)$. The orbits of this action have size 1, 2, 3 or 6.
	\par It is easy to see that there are no orbits of size 1. Any orbit of size 2 is of the form $\{x,x^{-1}\}$ for some $x\in S\subseteq(\Z/pq\Z)^{*}$ satisfying $x^{2}-x+1=0$. Then $x \pmod p$ satisfies the same equation in $(\Z/p\Z)^{*}$ i.e.~it is a primitive $6$th root of unity. However, in $(\Z/p\Z)^{*}$ there are at most two primitive sixth roots of unity. Since each element of $S$ corresponds to a unique element in $(\Z/p\Z)^{*}$, there is at most one orbit of size 2. Moreover, there is such an orbit if there is an $x\in S\subseteq(\Z/pq\Z)^{*}$ satisfying $x^{2}-x+1=0$; such an $x$ clearly has order $6$ in $(\Z/pq\Z)^{*}$.
	\par In a similar but easier way, if $2 \in S$, then also $2 \in S \cap (S+1)$ and there is precisely one orbit of size 3, namely $\left\{-1,\frac{1}{2},2\right\}$, and else there are no orbits of size 3. All other orbits have size 6. So, indeed $\left|S\cap(S+1)\right|\equiv3\delta+2\epsilon\pmod{6}$.
\end{proof}

\begin{corollary}\label{cor:not1mod3}
	Let $p$ be an odd prime, let $q$ be an odd integer and let $a\in(\Z/pq\Z)^*$ be such that $a\pmod{p}$ is a generator of $(\Z/p\Z)^*,\cdot$ and such that $a^{\frac{p-1}{2}}\equiv-1\pmod{pq}$. Then we have $\left|S_{pq}(a)\cap(S_{pq}(a)+1)\right|\not\equiv1\pmod{3}$.
\end{corollary}

\begin{remark}\label{rem:nomultiplesof3}
    From Corollary \ref{cor:not1mod3} it follows that $\left|S_{pq}(a)\cap(S_{pq}(a)+1)\right|\not\equiv-2\pmod{q}$ if $q$ is a multiple of 3, given a prime $p$ and an integer $a$ satisfying the conditions of Theorem \ref{th:constructionpq}. So, for any $q$ which is a multiple of 3, it is impossible to construct a Neumaier graph using the construction in Theorem \ref{th:constructionpq}.
\end{remark}

\subsection{A joint condition on \texorpdfstring{$p$}{p} and \texorpdfstring{$q$}{q}}\label{ssec:conditions}

As we mentioned before, it is our aim to prove that there is an infinite number of odd integers $q$ such that for each of them there is an infinite number of primes $p$ for which an integer $a$ exists, satisfying the conditions from Theorem \ref{th:constructionpq}. We will show this in Sections \ref{ssec:infinite5}, \ref{ssec:infinite7} and \ref{ssec:infiniteinf}. This section serves as an introduction to that, fixing some notation.
\par Consider a positive odd integer $q$, a prime number $p > q$ and let $r = \nu_2(p-1) \geq 1$ denote the $2$-valuation of $p-1$, i.e.\ $2^{r}\mid p-1$, but $2^{r+1}\nmid p-1$. Let $a \in (\Z/pq\Z)^\ast$ be such that $a^{(p-1)/2} = -1$, and let $\alpha \in \F_p^\ast=(\Z/p\Z)^*$ and $\beta \in (\Z / q\Z)^\ast$ denote the reductions $a$ modulo $p$ and $a$ modulo $q$, respectively. As before, it is assumed that $\alpha$ is a generator of $\F_p^\ast$. In Section~\ref{ssec:formula} we will give a formula for the cardinality of $S \cap (S+1)$ with $S = S_{pq}(a)$ in terms of Jacobi sums of order $n = \ord(\beta)$.

Let us first discuss a joint condition on $p$ and $q$ for there to exist such an element $a$, regardless of the value of $| \, S \cap (S+1) \, |$. Consider the factorization $q = \ell_1^{e_1} \cdots \ell_k^{e_k}$ of $q$ into powers of distinct (necessarily odd) primes $\ell_i$. For each $i$, let $\beta_i \in (\Z / \ell_i^{e_i} \Z)^\ast$ be the reduction of $\beta$ modulo $\ell^{e_i}$, and denote by $n_i$ its order. From
\[ \beta_i^{(p-1)/2} = -1 \]
it follows that $n_i \mid p-1$ and that $\nu_2(n_i) = \nu_2(p-1) = r$, independently of $i$ (in particular all $n_i$ are even). This is only possible if $p, q$ are such that $2^r \mid \varphi(\ell_i^{e_i}) = \ell_i^{e_i-1} (\ell_i - 1)$, or in other words such that
\begin{align}\label{eq:condition}
2^r \mid \ell_i - 1, \quad \text{for all $i = 1, \ldots k$}. 
\end{align}
For use below, we note that $n = \lcm(n_1, \ldots, n_k)$ then also satisfies $\nu_2(n) = r$ (in particular $n$ is even), so that 
\[ \beta_i^{n/2} = -1 \] 
for all $i$, which in turn implies that $\beta^{n/2} = -1$.

Condition~\eqref{eq:condition} is necessary, but also  sufficient. Indeed, if $p, q$ are such that $2^r \mid \ell_i - 1$ for all $i$, then we can choose any elements $\beta_i \in (\Z / \ell_i^{e_i} \Z)^\ast$ of order $2^rs_i$, with $s_i$ some odd common divisor of $p-1$ and $\varphi(\ell_i^{e_i}) $, and any generator $\alpha$ of $\F_p^\ast$, and combine them into an element $a \in (\Z / pq \Z)^\ast$ of the desired form, using the Chinese remainder theorem.

\subsection{Preliminaries on Jacobi sums}\label{ssec:jacobi}

For an odd prime number $p$, a \emph{character} mod $p$
is a group homomorphism $\chi : \F_p^\ast \to \C^\ast$. The image of $\chi$ is the group $\mu_n$ of $n$-th roots of unity, for some integer $n \geq 1$ dividing $p-1$ that we call the \emph{order} of $\chi$. Equivalently, the order of $\chi$ is just its order as an element of the character group (i.e., with respect to point-wise multiplication). If $n = 1$ then $\chi$ is said to be trivial. We always have $\chi(1) = 1$, and it is customary to extend the domain of $\chi$ to all of $\F_p$ by defining $\chi(0) = 0$, unless $\chi$ is trivial in which case one lets $\chi(0) = 1$.

If $\chi$ and $\lambda$ are two characters mod $p$, then the corresponding \emph{Jacobi sum} is defined to be
\[ J(\chi, \lambda) = \sum_{\substack{c, d \in \F_p \\ c + d = 1}} \chi(c) \lambda(d),\]
which we note is the complex conjugate of $J(\chi^{-1}, \lambda^{-1})$.
If $\varepsilon$ denotes the trivial character mod $p$, then we have the immediate rule
\begin{equation} \label{eq:trivialcharacter} 
J(\chi, \varepsilon) = \left\{ \begin{array}{ll} 0 & \text{if $\chi \neq \varepsilon$,} \\ p & \text{if $\chi = \varepsilon$,} \\ \end{array} \right.
\end{equation}
and it is not hard to check that 
\begin{equation} \label{eq:Jchichimin1}
  J(\chi, \chi^{-1}) = - \chi(-1)
\end{equation}
as soon as $\chi \neq \varepsilon$. More advanced identities can be found in~\cite[Ch.\,8]{irelandrosen}, to which we refer for a gentle introduction to Jacobi sums, and in~\cite[Ch.\,3]{berndtevanswilliams}, which contains explicit formulae for Jacobi sums involving characters of order $n \leq 8$ and $n = 10, 12, 16, 20, 24$. For the reader's convenience, let us include the cases $n = 2, 4, 6$. We denote the square roots of $-1$ by $\pm\mathbf{i}$.

\begin{example} If $\chi$ is a character of order $2$, then
\[ J(\chi, \chi) = J(\chi, \chi^{-1}) = - \chi(-1) = - \left( \frac{-1}{p} \right) = (-1)^{\frac{p+1}{2}}.\]
\end{example}

\begin{example} \cite[Section 3.2]{berndtevanswilliams} \label{ex:jacobiq5}
If $\chi$ is a character of order $4$, then necessarily $p \equiv 1 \pmod 4$. Let $g \in \F_p^\ast$ be such that $\chi(g) = \mathbf{i}$. There exist unique integers $x, y$ such that
\begin{equation} \label{eq:sumoftwosquares} p = x^2 + y^2, \qquad x \equiv - \left( \frac{2}{p} \right) \pmod 4, \qquad y \equiv x g^{\frac{p-1}{4}} \pmod p. 
\end{equation}
Then the values of $J(\chi^i, \chi^j)$ for $i,j = 1, 2, 3$ are as follows:
\begin{center}
\vspace{0.1cm}
\begin{tabular}{|c|c|c|c|}\hline
\diagbox{$i$}{$j$}&
     $1$      &      $2$     & $3$ \\ \hline
$1$ &  $(-1)^f (x + y\mathbf{i})$      &    $x + y\mathbf{i}$  &  $(-1)^{f + 1}$   \\ \hline
$2$ &   $x + y\mathbf{i}$  &      $-1$    &  $x - y\mathbf{i}$   \\ \hline
$3$ &   $(-1)^{f + 1}$ &  $x - y\mathbf{i}$   &  $(-1)^f (x - y\mathbf{i})$   \\ \hline
\end{tabular}
\vspace{0.15cm}
\end{center}
where $f = (p-1)/4$.
\end{example}

\begin{example} \cite[Section 3.1]{berndtevanswilliams} \label{ex:jacobiq7}
If $\chi$ is a character of order $6$, then we must have $p \equiv 1 \pmod 6$. Let $\zeta = e^{2 \pi \mathbf{i} / 6} = (1 + \mathbf{i}\sqrt{3})/2$ and let $g \in \F_p^\ast$ be such that $\chi(g) = \zeta$. There exist unique integers $x,y$ such that
\begin{equation} \label{eq:sum2ofsquares} p = x^2 + 3y^2, \qquad x \equiv - 1 \pmod 3, \qquad 3y \equiv (2g^{\frac{p-1}{3}} + 1)x \pmod p\;. 
\end{equation}
We further define
\[ \left\{ \begin{array}{lll} 
r = 2x, \, s = 2y, & u = 2x, \, v = 2y, & \text{if $y \equiv 0 \pmod 3$,} \\
r = - x + 3y, \, s = -x - y, & u = -x -3y, \, v = x - y,  & \text{if $y \equiv 1 \pmod 3$,} \\
r = - x - 3y, \, s = x - y, & u = -x + 3y, \, v = -x - y,  & \text{if $y \equiv 2 \pmod 3$,} \\ \end{array} \right.\]
where we note that $4p = r^2 + 3s^2 = u^2 + 3v^2$. The values of $J(\chi^i, \chi^j)$ for $i,j = 1, 2, 3, 4, 5$ are as follows:
\begin{center}
\vspace{0.1cm}
\begin{tabular}{|c|c|c|c|c|c|}\hline
\diagbox{$i$}{$j$}&
     $1$      &      $2$     & $3$ & $4$ & $5$ \\ \hline
$1$ &  $(-1)^{f} \frac{u + v \mathbf{i}\sqrt{3}}{2}$ & $x + y \mathbf{i}\sqrt{3}$ & $(-1)^f (x + y\mathbf{i}\sqrt{3})$ & $\frac{u + v \mathbf{i}\sqrt{3}}{2}$ & $(-1)^{f+1}$ \\ \hline
$2$ &  $x + y\mathbf{i}\sqrt{3}$ & $\frac{r + s \mathbf{i}\sqrt{3}}{2}$ & $x + y\mathbf{i}\sqrt{3}$ & $-1$ & $\frac{u - v \mathbf{i}\sqrt{3}}{2}$\\ \hline
$3$ &  $(-1)^{f} (x + y\mathbf{i}\sqrt{3})$ & $x + y\mathbf{i}\sqrt{3}$ & $(-1)^{f + 1}$ & $x - y\mathbf{i}\sqrt{3}$ & $(-1)^f (x - y\mathbf{i}\sqrt{3})$ \\ \hline
$4$ &  $\frac{u + v \mathbf{i}\sqrt{3}}{2}$ & $-1$ & $x - y\mathbf{i}\sqrt{3}$ & $\frac{r - s \mathbf{i}\sqrt{3}}{2}$ & $x - y\mathbf{i}\sqrt{3}$ \\ \hline
$5$ &  $(-1)^{f + 1}$ & $\frac{u - v \mathbf{i}\sqrt{3}}{2}$ & $(-1)^f (x - y\mathbf{i}\sqrt{3})$ & $x - y\mathbf{i}\sqrt{3}$ & $(-1)^{f} \frac{u - v \mathbf{i}\sqrt{3}}{2}$ \\ \hline 
\end{tabular}
\vspace{0.15cm}
\end{center}
\end{example}
\noindent where $f = (p-1)/6$.

\subsection{A formula for \texorpdfstring{$| \, S \cap (S + 1) \, |$}{|S cap (S+1)|}} \label{ssec:formula}

We can convert the natural surjection
$\xi : \F_p^\ast \to \langle \beta \rangle : \alpha^j \mapsto \beta^j$ into an order-$n$ character $\chi$ by composing it with the isomorphism
\[ \psi : \langle \beta \rangle \to \mu_n : \beta^j  \mapsto e^{ 2 \pi \mathbf{i} j / n}.  \]
Recall from Section~\ref{ssec:conditions} that $\beta^{n/2} = -1$, hence $\psi(-1) = -1$, so that 
\[ \chi(-1) = \psi\left(\xi\left(\alpha^{(p-1)/2}\right)\right) = \psi\left(\beta^{(p-1)/2}\right) = \psi(-1) = -1\;.  \]
The proof below makes a frequent use of this fact. For a complex number $z$ we denote the real part by $\Re(z)$.

\begin{theorem}\label{thm:formula}
Writing $B = \{ \, b \in \langle \beta \rangle \, | \, b - 1 \in \langle \beta \rangle \, \}$, we have
\begin{equation} \label{eq:formulaforScapS+1} 
| \, S \cap (S + 1) \, | \ = \ \frac{1}{n^2} \left( (p + 1) \left| B \right| +  
\sum_{1 \leq i \leq j < n-i} 2(2 - \delta_{i,j})\Re(c_{i,j} J(\chi^i, \chi^j)) \right), 
\end{equation}
where $c_{i,j} = \sum_{b \in B} \psi(b)^{-i} \psi(1- b)^{-j}$
and $\delta_{i,j}$ is the Kronecker symbol. 
\end{theorem}
\begin{proof}
Under the Chinese remainder theorem, the set $S$ corresponds to
\[ \{ \, (c, \xi(c)) \, | \, c \in \F_p^\ast \, \} \ \subseteq \ \F_p \times (\Z / q \Z)\;, \]
so we have
\begin{align*}
 | \, S \cap (S + 1) \, | \ & = \
\left| \{ \, c \in \F_p \setminus \{0, 1\} \, | \, \exists b \in B \, : \,
 \xi(c-1) = b - 1 \text{ and } \xi(c) = b \, \} \right| \\
 & = \sum_{b \in B} \left| \{ \, c \in \F_p \setminus \{0, 1\} \, | \,  \chi(c-1) = \psi(b - 1) \text{ and } \chi(c) = \psi(b) \, \}  \right|.
\end{align*}
Each summand of the right-hand side can be rewritten as
\begin{equation} \label{eq:summand} 
\sum_{ c \in \F_p \setminus \{0, 1\} } 
\left( \frac{\psi(b)}{n} \prod_{\substack{\zeta \in \mu_n \\ \zeta \neq \psi(b)}} ( \chi(c) - \zeta ) \right)  \left( \frac{\psi(b - 1)}{n}  \prod_{\substack{\zeta \in \mu_n \\ \zeta \neq \psi(b - 1)}} ( \chi(c - 1) - \zeta ) \right),
\end{equation}
where we have used that
\[ \prod_{\substack{\zeta \in \mu_n \\ \zeta \neq \psi(b)}} (\psi(b) - \zeta ) = \psi(b)^{n-1} \prod_{\substack{\zeta \in \mu_n \\ \zeta \neq \psi(b)}} ( 1 - \zeta \psi(b)^{-1} ) = \psi(b)^{-1} \prod_{\substack{\zeta \in \mu_n \\ \zeta \neq 1}} (1 - \zeta) = n \psi(b)^{-1},  \]
and likewise for $\psi(b-1)$; to see the last equality, evaluate the polynomial $(X^n - 1)/(X - 1) = X^{n-1} + \ldots + 1$ at $1$. 

We can let the sum in~\eqref{eq:summand} range over every $c \in \F_p$ without affecting it. Indeed, the contribution of $c = 1$ is zero since $\chi(1) = 1$ and $\psi(b) \neq 1$ (because $1 \notin B$), and similarly the contribution of $c = 0$ is zero because $\chi(-1) = -1$ and $\psi(b-1) \neq -1$ (because $0 \notin B$). Writing $d = 1 - c$, one sees that expression~\eqref{eq:summand} then becomes 
\begin{align}\label{eq:summand2}
    \frac{\psi(b)\psi(b-1)}{n^2} \sum_{\substack{c, d \in \F_p \\ c + d = 1}}  \left(\prod_{\substack{\zeta \in \mu_n \\ \zeta \neq \psi(b) }} (\chi(c) - \zeta)\right)\left( (-1)^{n-1} \prod_{\substack{\zeta \in \mu_n \\ \zeta \neq \psi(1 - b) }} (\chi(d) - \zeta)\right) \nonumber \\
    =\frac{\psi(b)\psi(1-b)}{n^2} \sum_{\substack{c, d \in \F_p \\ c + d = 1}}  \prod_{\substack{\zeta \in \mu_n \\ \zeta \neq \psi(b) }} (\chi(c) - \zeta) \prod_{\substack{\zeta \in \mu_n \\ \zeta \neq \psi(1 - b) }} (\chi(d) - \zeta), 
\end{align}
which we can view as the evaluation  of
\[ \frac{\psi(b)\psi(1-b)}{n^2} \sum_{\substack{c, d \in \F_p \\ c + d = 1}} \frac{(X^n - 1)(Y^n - 1)}{(X - \psi(b))(Y - \psi(1 - b))}  \]
at $X = \chi(c), Y = \chi(d)$. One checks that
\[ \frac{\psi(b)\psi(1-b)(X^n - 1)(Y^n - 1)}{(X - \psi(b))(Y - \psi(1 - b))} = \sum_{i,j = 0}^{n-1} \psi(b)^{-i} \psi(1 - b)^{-j} X^i Y^j,\]
allowing us to rewrite~\eqref{eq:summand2} as
\[ \frac{1}{n^2} \sum_{\substack{c, d \in \F_p \\ c + d = 1}}
\sum_{i,j = 0}^{n-1} \psi(b)^{-i} \psi(1 - b)^{-j} \chi^i(c) \chi^j(d) = \frac{1}{n^2} \sum_{i , j = 0}^{n-1} \psi(b)^{-i} \psi(1 - b)^{-j} J(\chi^i, \chi^j). \]
Note that the terms for which $i = 0$ or $j = 0$ sum up to $p$, in view of~\eqref{eq:trivialcharacter}. Using that 
\[ J(\chi^i, \chi^{n-i}) = -\chi^i(-1) = (-1)^{i+1}, \]
which follows from~\eqref{eq:Jchichimin1}, the terms for which $i + j = n$ can be seen to sum up to $1$. Indeed,
\begin{align*}
    \sum_{i=1}^{n-1} \psi(b)^{-i} \psi(1 - b)^{i-n} J(\chi^i, \chi^{n-i})&=\sum_{i=1}^{n-1} \psi(b)^{-i} \psi(1 - b)^{i-n}(-1)^{i+1}\\
    &=-\sum_{i=1}^{n-1}\left(-\psi(b)^{-1}\psi(1 - b)\right)^{i}\\
    &=1-\sum_{i=0}^{n-1}\left(-\psi(b)^{-1}\psi(1 - b)\right)^{i}\\
    &=1-\frac{\left(-\psi(b)^{-1}\psi(1 - b)\right)^{n}-1}{\left(-\psi(b)^{-1}\psi(1 - b)\right)-1}=1.
\end{align*}
Altogether, we find that
\begin{align*}
    | \, S \cap (S + 1) \, | &= \sum_{b\in B}\frac{1}{n^2} \sum_{i , j = 0}^{n-1} \psi(b)^{-i} \psi(1 - b)^{-j} J(\chi^i, \chi^j) \\
    &= \frac{1}{n^2} \sum_{b\in B} \left( p + 1 + \sum_{\substack{1 \leq i, j \leq n - 1 \\ i + j \neq n}} \psi(b)^{-i} \psi(1 - b)^{-j} J(\chi^i, \chi^j) \right) \\
    &=\frac{1}{n^2} \left( (p + 1)|B| + \sum_{\substack{1 \leq i, j \leq n - 1 \\ i + j \neq n}} c_{i,j} J(\chi^i, \chi^j) \right)
\end{align*}
with $c_{i,j}$ as in the statement of the theorem. Next, using $-1 \in \langle \beta \rangle$, one checks that $b \mapsto 1 - b$ is an involution of $B$, from which it follows that $c_{i,j} = c_{j,i}$ and hence $c_{i,j} J(\chi^i, \chi^j) = c_{j,i} J(\chi^j, \chi^i)$ for all $i,j$.
The theorem then follows because $c_{i,j} J(\chi^i, \chi^j)$ is the complex conjugate of $c_{n-i,n-j} J(\chi^{n-i}, \chi^{n-j})$.
\end{proof}

\begin{remark}\label{rem:emptyB}
From the previous theorem it follows immediately that $| \, S \cap (S + 1) \, | = 0$ if $B=\emptyset$. Actually, one can see this already in the beginning of the proof. In some cases it is easy to show that $B=\emptyset$, see Examples \ref{ex:exampleq5} and~\ref{ex:exampleq7}, as well as Theorem \ref{th:withFermatprime}.
\end{remark}

\begin{example}
 If $q = 3$ then condition~\eqref{eq:condition} amounts to $r = 1$, therefore we should restrict to $p \equiv 3 \pmod 4$. The only option for $\beta$ is $-1$. Then $n=2$ and the corresponding set $B$ is just the singleton $\{ -1 \}$.  
 The theorem yields $| \, S \cap (S + 1) \, | = (p + 1)/4$. It is easy to check that this is never congruent to $-2 \pmod 3$, so it is impossible to construct a Neumaier graph using the construction in Theorem \ref{th:constructionpq}. But this we already knew from Section~\ref{ssec:nonadmissible}.
\end{example}

\begin{example} \label{ex:exampleq5}
  If $q = 5$ then $r = 1$ or $r = 2$. If $r = 1$, or in other words $p \equiv 3 \pmod 4$, then $\beta = -1$, but in this case $B = \emptyset$ so that $| \, S \cap (S + 1) \, | = 0$. Thus we focus on the case $r = 2$, i.e., the case $p \equiv 5 \pmod 8$. Then $\beta \in \{ \pm 2 \}$, and $B = \{ -2, -1, 2 \}$.
  \par For $\beta = 2$, the values of $\psi(b)$ are $- \mathbf{i}$, $-1$, $\mathbf{i}$, and those of $\psi(1-b)$ are $-\mathbf{i}$, $\mathbf{i}$, $-1$, for $b=-2$, $-1$ and $2$ respectively. We find that
  \[
    | \, S \cap (S + 1) \, | = \frac{1}{16} \left( \, 3p + 3 + 2\Re((-1 + 2 \mathbf{i})J(\chi, \chi)) + 4\Re((1 -2\mathbf{i})J(\chi, \chi^2)) \right). \]
  Letting $x,y$ be as in~\eqref{eq:sumoftwosquares}, i.e.,
  \[ p = x^2 + y^2, \qquad x \equiv - \left( \frac{2}{p} \right) \equiv 1 \pmod 4, \qquad y \equiv x \alpha^{\frac{p-1}{4}} \pmod p, \]
  we find from Theorem \ref{thm:formula} that
  \[ | \, S \cap (S + 1) \, | = \frac{3}{16} (p + 1 + 2x + 4y), \]
  using the results on Jacobi sums in the table in Example \ref{ex:jacobiq5}. Note our usage of $p \equiv 5 \pmod 8$ in several steps. For $\beta = -2$, an analogous computation shows that $| \, S \cap (S + 1) \, | = \frac{3}{16} (p + 1 + 2x - 4y)$.
\end{example}

\begin{example} \label{ex:exampleq7}
If $q = 7$ then $r = 1$, so $p \equiv 3 \pmod 4$. The possible values of $n = \ord(\beta)$ are $2$ and $6$. If $n =2$ then $\beta = -1$ and $B = \emptyset$, hence $| \, S \cap (S + 1) \, | = 0$. Therefore we assume $ n=6$, which implies that $\beta$ is a generator of $\F_7^\ast$ and that $p \equiv 1 \pmod 6$; consequently $p \equiv 7 \pmod{12}$. We focus on $\beta = 3$, leaving the analogous case $\beta = -2$ for the reader.
\par We have $B = \{ -3, -2, -1, 2, 3\}$, and we list the values of $\psi(b)$ and $\psi(1-b)$ for all $b\in B$:
\begin{center}
    \begin{tabular}{c|ccccc}
      $b$ & $-3$ & $-2$ & $-1$ & $2$ & $3$ \\\hline
      $\psi(b)$ & $-\zeta$ & $-\zeta^2$ & $-1$ & $\zeta^2$ & $\zeta$\\
      $\psi(1-b)$ & $-\zeta$ & $\zeta$ & $\zeta^2$ & $- 1$ & $-\zeta^2$
    \end{tabular}.
\end{center}
Theorem \ref{thm:formula} yields that $| \, S \cap (S + 1) \, |$ equals
 \begin{multline*}
     \frac{1}{36} \left[ 5p + 5 + 2 \Re\left( \tfrac{5 + \mathbf{i}\sqrt{3}}{2} J(\chi, \chi)\right) + 4 \Re\left( (2 - \mathbf{i}\sqrt{3} )J(\chi, \chi^2)\right)  + 4 \Re\left( (-2 + \mathbf{i}\sqrt{3} )J(\chi, \chi^3)\right) \right.\\\left. +4 \Re\left(\tfrac{-5-\mathbf{i}\sqrt{3}}{2} J(\chi, \chi^4)\right) + 2 \Re\left( \tfrac{1 + 3\mathbf{i}\sqrt{3}}{2} J(\chi^2, \chi^2)\right)  + 4\Re\left((2 - \mathbf{i}\sqrt{3}) J(\chi^2, \chi^3)\right) \right] 
\end{multline*}
which can be rewritten as
\begin{align*}
    &\frac{1}{36} \left( 5p + 5 + 2\left(\tfrac{-5u + 3v}{4}\right) + 4(2x + 3y) + 4(2x + 3y)   + 4\left(\tfrac{-5u + 3v}{4}\right) + 2 \left(\tfrac{r - 9s}{4}\right) + 4 \left(2x + 3y\right) \right)\\
    &=\frac{1}{36}\left(5p+5+24x+36y+\frac{r-9s-15u+9v}{2}\right)
\end{align*}
using the results on Jacobi sums in the table in Example \ref{ex:jacobiq7}, where $x$, $y$ are as in~\eqref{eq:sum2ofsquares} and where $r,s, u, v$ are defined correspondingly (see Example \ref{ex:jacobiq7}, where we take $g = \alpha$). This leads to the conclusion that
\begin{equation} \label{eq:formulaScapSp1q7}
 36 \cdot | \, S \cap (S + 1) \, | = \left\{ \begin{array}{ll} 
 5p + 5 + 10x + 36y & \text{if $y \equiv 0 \pmod 3$}, \\ 
 5p + 5 + 40x + 60y & \text{if $y \equiv 1 \pmod 3$}, \\
 5p + 5 + 22x + 12y & \text{if $y \equiv 2 \pmod 3$}. \\
 \end{array} \right.
\end{equation}
\end{example}

\begin{example}\label{ex:examplen6}   
Let $q = \ell_1^{e_1} \cdots \ell_k^{e_k} > 7$ be such that all its prime divisors $\ell_i$ satisfy $\ell_i\equiv1\pmod{6}$. We can choose $r = 1$, so $p \equiv 3 \pmod 4$. For each $i$, consider a primitive $6$-th root of unity $\beta_i \in (\Z / \ell_i^{e_i} \Z)^\ast$, i.e., we let $\beta_i$ be one of the two solutions to $X^2 - X + 1 = 0$. To see why there are two solutions: there are two solutions modulo $\ell_i$ because $\left(\frac{-3}{\ell_{i}}\right)=\left(\frac{\ell_{i}}{3}\right)=1$ since $\ell_i \equiv 1 \pmod 6$, and each of these solutions lifts to a unique solution modulo $\ell_i^{e_i}$ by Hensel's lemma~\cite[Thm.\,2.23]{numbertheory}. Using the Chinese remainder theorem, we combine these $\beta_i$'s into a single element $\beta \in (\Z / q \Z)^\ast$. It is clearly again of order $n = 6$, and it satisfies 
\begin{equation} \label{eq:betarecursion}
  \beta^2 = \beta - 1\;.
\end{equation}
Our choice of $\beta$ implies that $p \equiv 1 \pmod 6$; consequently $p \equiv 7 \pmod{12}$. 

\par Using~\eqref{eq:betarecursion} one checks that $\left\langle\beta\right\rangle=\left\{\beta,\beta-1,-1,-\beta,1-\beta,1\right\}$, and from $q > 7$ one sees that $B = \{\beta,1-\beta\}$. We immediately find  $\psi(\beta)=\psi(1-(1-\beta))=\zeta$ and $\psi(1-\beta)=\zeta^{-1}=-\zeta^{2}$. From Theorem \ref{thm:formula} we get:
\begin{align*}
    | \, S\cap(S+1) \, |&=\frac{1}{36}\left(2p+2+2\cdot2\cdot\left(-\frac{u}{2}\right)+4\cdot1\cdot x+4\cdot(-1)\cdot(-x)\right.\\ &\left. \qquad\qquad+4\cdot(-2)\cdot\frac{u}{2}+2\cdot2\cdot\frac{r}{2}+4\cdot1\cdot x\right)\\
    &=\frac{1}{36}\left(2p+2+12x+2r-6u\right)
\end{align*}
using the results on Jacobi sums in the table in Example \ref{ex:jacobiq7}, where $x$, $y$ are as in~\eqref{eq:sum2ofsquares} and where $r,s, u, v$ are defined correspondingly (see Example \ref{ex:jacobiq7}, where we take $g = \alpha$). This leads to the conclusion that
\begin{align}\label{eq:formulaScapSp1n6}
    36\cdot|S\cap(S+1)|=
    \begin{cases}
    2p+2+4x & \text{if }y\equiv0\pmod3,\\
    2p+2+16x+24y & \text{if }y\equiv1\pmod3,\\
    2p+2+16x-24y & \text{if }y\equiv2\pmod3.\\
    \end{cases}
\end{align}
\end{example}

A \emph{Fermat prime} is a prime of the form $2^{2^n}+1$ for some integer $n$. The only known Fermat primes are 3, 5, 17, 257 and 65537. It is conjectured there are no others. 

\begin{theorem}\label{th:withFermatprime}
	Let $p$ be an odd prime, let $q$ be an odd integer and let $a\in(\Z/pq\Z)^*$ be such that $a\pmod{p}$ is a generator of $(\Z/p\Z)^*,\cdot$ and such that $a^{\frac{p-1}{2}}\equiv-1\pmod{pq}$. Let $q=\prod_{i=1}^{m}\ell_{i}^{e_{i}}$ be the prime power decomposition of $q$. If there is an $i$ such that $\ell_{i}\geq5$ is a Fermat prime, and there is a $j$ such that $\ell_{j}\equiv3\pmod{4}$, then $\left|S_{pq}(a)\cap(S_{pq}(a)+1)\right|=0$.
\end{theorem}
\begin{proof}
    From \eqref{eq:condition} and $\ell_{j}\equiv3\pmod{4}$ it follows immediately that $r=1$. Thus the order of $\beta_i \in (\Z / \ell_i^{e_i}\Z)^\ast$, the reduction of $a$ modulo $\ell_i^{e_i}$, equals $2s_i$ for some odd $s_i$. Further reducing mod $\ell_i$, we find an element $\overline{\beta}_i \in (\Z / \ell_i\Z)^\ast$ whose order divides $2s_i$. But since $\ell_i$ is a Fermat prime, the order of $(\Z / \ell_i\Z)^\ast$ is a power of $2$. Hence $\ord(\overline{\beta}_i)$ is equal to $1$ or $2$, or in other words $\overline{\beta}_i$ equals $-1$ or $1$.
    
    Now assume that $B = \{ \, b \in \langle \beta \rangle \, | \, b - 1 \in \langle \beta \rangle \, \}$ is non-empty, i.e.\ we have $\beta^r - 1 = \beta^s$ for certain exponents $r, s$. Reducing mod $\ell_i$ yields $\overline{\beta}_i^r - 1 = \overline{\beta}_i^s$. But given that $\overline{\beta}_i = \pm 1$ and $\ell_i \geq 5$, this is impossible. So $B = \emptyset$ and the theorem follows from Theorem \ref{thm:formula} and Remark \ref{rem:emptyB}.
\end{proof}


\begin{remark}\label{rem:nowithFermatprime}
    From Theorem \ref{th:withFermatprime} it follows that $\left|S_{pq}(a)\cap(S_{pq}(a)+1)\right|\not\equiv-2\pmod{q}$ if $q=p'q'$ with $p'$ a Fermat prime and $q'$ having a prime factor $p''\equiv3\pmod{4}$, given a prime $p$ and an integer $a$ satisfying the conditions of Theorem \ref{th:constructionpq}. So, for any such $q$ it is impossible to construct a Neumaier graph using the construction in Theorem \ref{th:constructionpq}. The five smallest values of $q$ that have such a decomposition, and that are not multiples of 3, are 35, 55, 95, 115, and 119.
\end{remark}



\subsection{An infinite family of Neumaier graphs for \texorpdfstring{$q = 5$}{q=5}}\label{ssec:infinite5}

We can now explain why there exist infinitely many prime numbers $p$ for which there exists an $a \in (\Z / 5p\Z)^\ast$
meeting the conditions from Theorem~\ref{th:constructionpq} and such that $| \, S \cap (S+1) \, | \equiv -2 \pmod 5$. This argument mainly relies on the Gaussian integer analogue of a celebrated result by Dirichlet~\cite[Sect.\,V.6]{neukirch} which states that, for any integer $m \neq 0$ and any integer $a$ that is coprime to $m$, there exist infinitely many prime numbers $p \equiv a \pmod m$. 
The analogue for the Gaussian integers $\Z[\mathbf{i}]$ and for Eisenstein integers $\Z[\zeta]$ is as follows.

\begin{theorem} \label{thm:dirichlet}
Let $R = \Z[\mathbf{i}]$ or $R = \Z[\zeta]$ and consider $m \in R \setminus \{0\}$. Let $z \in R$ be coprime with $m$. Then there exist infinitely many prime elements $\pi \in R$ such that $m \mid \pi - z$. 
\end{theorem}
\begin{proof}
This follows from~\cite[Thm.\,V.6.2]{neukirch} or~\cite[Prop.\,28.10]{MITlecturenotes},
applied to the modulus $(m)$ for the number field $K = \text{Frac}(R)$.
\end{proof}


\begin{theorem} \label{thm:infinitelyq5}
There exist infinitely many prime numbers $p$ for which there exists an $a \in (\Z / 5p \Z)^\ast$ meeting the requirements from Theorem~\ref{th:constructionpq} and for which $S = S_{5p}(a)$ satisfies $| \, S \cap (S+1) \, | \equiv -2 \pmod 5$. 
\end{theorem}
\begin{proof}
We apply Theorem~\ref{thm:dirichlet} to $R = \Z[\mathbf{i}]$ with $m = 20$ and $z = 5  + 6\mathbf{i}$, which one verifies to be coprime to each other (it suffices to check that $\gcd(z \overline{z}, 20) = \gcd(5^2 + 6^2, 20) = 1$), to conclude that there exist infinitely many Gaussian primes $\pi$ such that
\begin{equation} \label{eq:gaussiandirichlet} 
  20 \, \mid \, \pi - (5 + 6 \mathbf{i} ).
\end{equation}
Recall that, up to multiplication with a unit of $\Z[\mathbf{i}]$, i.e., up to multiplication with $\pm 1, \pm \mathbf{i}$, all Gaussian primes $\pi$ are either integer primes $p \equiv 3 \pmod 4$, or of the form $x + y \mathbf{i}$ for integers $x,y$ such that $p = \pi \overline{\pi} = x^2 + y^2$ is an integer prime; in the latter case we necessarily have $p = 2$ or $p \equiv 1 \pmod 4$. 

Writing $\pi = x + y\mathbf{i}$, one sees from~\eqref{eq:gaussiandirichlet} that $x \equiv 5 \pmod{20}$ and $y \equiv 6 \pmod{20}$. In particular $x$ and $y$ are non-zero, hence $\pi$ cannot be of the form $\pm p, \pm p\mathbf{i}$ for some integer prime $p \equiv 3 \pmod 4$. Thus we must be concerned with a Gaussian prime of the second kind: $x^2 + y^2$ is a prime $p \equiv 1 \pmod 4$ (indeed, the case $p = 2$ is easily ruled out as well).

Since $x \equiv 1 \pmod 4$ and $y \equiv 2 \pmod 4$, we in fact know that $p = x^2 + y^2 \equiv 5 \pmod 8$.  Let $\alpha$ be a generator of $\F_p^\ast$ satisfying $y = x \alpha^{(p-1)/4} \pmod p$; such a generator indeed exists because $y/x$ is a primitive $4$-th root of unity in $\F_p^\ast$, being a square root of $y^2/x^2 \equiv -1$. Let $\beta = 2 \in \F_5^\ast$ and combine it with $\alpha$ into an element $a \in (\Z / 5p \Z)^\ast$ using the Chinese remainder theorem. Then, by Example~\ref{ex:exampleq5}, the corresponding set $S = S_{5p}(a)$ satisfies:
\[ | \, S \cap (S + 1) \, | = \frac{3}{16}(p + 1 + 2x + 4y) \equiv \frac{3}{1}(0^2 + 1^2 + 1 + 2\cdot 0 + 4 \cdot 1) \equiv -2 \pmod 5, \]
as wanted.

Since this construction applies to every Gaussian prime satisfying~\eqref{eq:gaussiandirichlet}, of which there is an infinite number, we indeed obtain the existence of infinitely many primes $p$ with the desired property.
\end{proof}

\begin{remark} \label{rmk:densityq5}
Note that we could have arrived at the same conclusion using other congruence classes mod $20$, rather than that of $5 + 6\mathbf{i}$.
Indeed, considering the congruence class of $z = z_1 + z_2 \mathbf{i}$ mod $20$, the above reasoning applies as soon as $z_1 \equiv 1 \pmod 4$, $z_2 \equiv 2 \pmod 4$, $\gcd(z_1^2 + z_2^2, 20) = 1$ and $\tfrac{3}{16}(z_1^2 + z_2^2 + 2z_1 + 4z_2) \equiv -2 \pmod 5$.
The reader can check that, besides $5 + 6\mathbf{i}$, the congruence classes of $1 + 14\mathbf{i}$, $13 + 10\mathbf{i}$, 
$17 + 2\mathbf{i}$ mod $20$ satisfy these conditions, and this list is exhaustive.

Let $P_5$ denote the set of prime numbers $p$ with the requested properties, i.e., for which there exists an element $a \in (\Z / 5p \Z)^\ast$ meeting the requirements from Theorem~\ref{th:constructionpq} and for which the corresponding set $S$ satisfies $| \, S \cap (S+1) \, | \equiv -2 \pmod 5$. We claim that all $p \in P_5$ arise as the norm of a Gaussian prime $\pi = x + y\mathbf{i}$ that belongs to one of the above congruence classes modulo $20$. As before, let $\alpha \in \F_p^\ast$ and $\beta \in \F_5^\ast$ denote the reductions of $a$ modulo $p$ and modulo $5$, respectively. From Example~\ref{ex:exampleq5} we know that $p$ is necessarily congruent to $5 \pmod 8$, hence of the form $x^2 + y^2$ with $x \equiv 1 \pmod 4$ and $y \equiv 2 \pmod 4$. By changing the sign of $a$ if needed we can assume that $\beta = 2$, and by changing the sign of $y$ if needed we can assume that $y\equiv x\alpha^{(p-1)/4} \pmod p$. From Example~\ref{ex:exampleq5} it then follows that $\frac{3}{16}(x^2 + y^2 + 2x + 4y) \equiv -2 \pmod 5$. This proves the claim.

We can use this to argue that the set $P_5$ has natural density
\[ \delta(P_5) = \lim_{X \to \infty} \frac{ \left| \{ \, \text{prime numbers $p \leq X$} \, | \, p \in P_5 \, \} \right|  }{ \left| \{ \, \text{prime numbers $p \leq X$} \, \} \right|   } = \frac{7}{64}. \]
Indeed, a refinement of Theorem~\ref{thm:dirichlet} states that, for each $z$ in the above list, the density of prime ideals
of $\Z[\mathbf{i}]$ having a generator $\pi$ that satisfies $20 \mid \pi - z$ is $1/32$, where the denominator $32$ arises as the size of the ray class group of $\Q(\mathbf{i})$ for modulus $(20)$. Explicitly,
\begin{equation} \label{eq:chebotarev} 
\lim_{X \to \infty} \frac{\left| \{ \, \text{prime ideals $(\pi) \subseteq \Z[\mathbf{i}]$ of norm $p = \pi \overline{\pi} \leq X$} \, | \, \pi \equiv z \pmod{20} \, \} \right|  }{ \left| \{ \, \text{prime ideals $(\pi) \subseteq \Z[\mathbf{i}]$ of norm $p = \pi \overline{\pi} \leq X$} \, \} \right|   }  = \frac{1}{32}
\end{equation}
(see~\cite[Prop.\,26.10]{MITlecturenotes}, and see~\cite[Rmk.\,26.12]{MITlecturenotes} for why we can use the natural density instead of the Dirichlet density). 

Now observe that the limit in~\eqref{eq:chebotarev} is not affected when replacing the denominator with the cardinality ${ \left| \{ \, \text{prime numbers $p \leq X$} \, \} \right|   }$. Indeed, we can ignore the unique prime ideal of norm $2$ and rewrite this denominator as 
\begin{align*} 
 2 \cdot \left| \{ \, \text{primes numbers $p \leq X$} \, | \, p \equiv 1 \pmod 4 \, \} \right| + \ \big| \{ \, \text{prime numbers $p \leq \sqrt{X}$} \, | \, p \equiv 3 \pmod 4 \, \} \big|\;.
\end{align*}
Then the observation follows because the prime numbers are equidistributed among the residue classes $1 \pmod 4$ and $3 \pmod 4$.
As for the numerator, if $z \neq 13 + 10\mathbf{i}$ then, subject to the congruence $\pi \equiv z \pmod{20}$, one sees that $(\pi)$ is uniquely determined by $p = \pi \overline{\pi}$. This is different for $z = 13 + 10\mathbf{i}$, where both $(\pi)$ and $(\overline{\pi})$ contribute to the numerator. We conclude that the numerator of $\delta(P_5)$ is the sum of the numerators of~\eqref{eq:chebotarev} for $z = 5 + 6\mathbf{i}, 1 + 14\mathbf{i}, 17 + 2\mathbf{i}$ and half the numerator of~\eqref{eq:chebotarev} for $z = 13 + 10\mathbf{i}$, from which the density $1/32 + 1/32 + 1/32 + 1/64 = 7/64$ follows.
\end{remark}

\begin{example}
The Gaussian prime $\pi = -15 - 14\mathbf{i}$ of norm $p = \pi \overline{\pi} = 421$ satisfies~\eqref{eq:gaussiandirichlet}.
The generator $\alpha = 2$ of $\F_{421}^\ast$ meets the requirement $\alpha^{(p-1)/4} \equiv y/x \equiv (-14)/(-15) \pmod p$. With $\beta = 2 \in \F_5^\ast$ this combines into $a = 2 \in (\Z / 2105\Z)^\ast$. The corresponding set $S=S_{2105}(2)$ satisfies $| \, S \cap (S + 1) \, | = \frac{3}{16}(p + 1 -2 \cdot 15 - 4 \cdot 14) = 63 \equiv -2 \pmod 5$.
\end{example}

\begin{remark}\label{rem:infinitestrictq5}
    From Theorem \ref{thm:infinitelyq5} it follows that there are infinitely many primes $p$ for which an integer $a$ exists such that $S = S_{5p}(a)$ satisfies $| \, S \cap (S+1) \, | \equiv -2 \pmod 5$. But, using the notation from Example~\ref{ex:exampleq5}, it also follows that
    \[ | \, S \cap (S + 1) \, | = \frac{3}{16} (p + 1 + 2x + 4y)>\frac{3}{16} (p + 1 -4\sqrt{p})=\frac{3}{16} ((\sqrt{p}-2)^{2}-3)\;. \]
    Consequently, if $p\geq41$, then $t=\frac{|S \cap (S + 1)| +2}{5}>1$. So, the Neumaier graphs that we find using the construction in Theorem~\ref{th:constructionpq} are strictly Neumaier by Theorem \ref{th:mainconstructionstrictly}. Hence, the construction in Theorem \ref{th:constructionpq} produces infinitely many strictly Neumaier graphs for $q=5$.
\end{remark}

\subsection{An infinite family of Neumaier graphs for \texorpdfstring{$q = 7$}{q=7}}\label{ssec:infinite7}

In this section we prove a result for $q = 7$, which is analogous to Theorem \ref{thm:infinitelyq5}.

\begin{theorem}\label{thm:infinitely7}
There exist infinitely many prime numbers $p$ for which there exists an $a \in (\Z / 7p \Z)^\ast$ meeting the requirements from Theorem~\ref{th:constructionpq} and for which $S = S_{7p}(a)$ satisfies $| \, S \cap (S+1) \, | \equiv -2 \pmod 7$.
\end{theorem}
\begin{proof}
Here, we apply Theorem~\ref{thm:dirichlet} to conclude that there exist infinitely many Eisenstein primes $\pi \in \Z[\zeta]$ such that
\begin{equation} \label{eq:eisensteindirichlet}
   84 \, \mid \, \pi - (3 + 10\zeta).
\end{equation}
Up to multiplication with one of the six units 
$\pm 1$, $\pm \zeta$, $\pm \zeta^2$
of $\Z[\zeta]$, the Eisenstein primes $\pi$ are either integer primes $p\equiv 2 \pmod 3$, or of the form 
$c + d\zeta$ for integers $c,d$ such that $p = \pi \overline{\pi} = c^2 + cd + d^2$ is an integer prime, in which case we necessarily have $p = 3$ or $p \equiv 1 \pmod 3$. 

Writing $\pi = c + d\zeta$, we get from~\eqref{eq:eisensteindirichlet} that $c \equiv 3 \pmod {84}$ and $d \equiv 10 \pmod {84}$. In particular $c \neq 0$, $d \neq 0$ and $c \neq -d$, so that $\pi$ cannot be of the form $\pm p$, $\pm p \zeta$ or $\pm p \zeta^2 = \mp p \pm p \zeta$ for some integer prime $p \equiv 2 \pmod 3$. Thus we are concerned with an Eisenstein prime of the second kind: $c^2 + cd + d^2$ is a prime $p \equiv 1 \pmod 3$ (indeed, the case $p = 3$ is easily ruled out as well).

We now define $x = c + d/2$ and $y = d/2$, which are integers because $d \equiv 10 \pmod{84}$ implies that $d$ is even. Note that
\[ \pi = c + d \zeta = c + d \,\frac{1 + \mathbf{i}\sqrt{3}}{2} =  x + y \mathbf{i} \sqrt{3} \]
and that $x \equiv 3 + 5 \equiv 8 \pmod{42}$ and $y \equiv 5 \pmod{42}$.

In particular it follows that $x \equiv 0 \pmod 2$ and $y \equiv 1 \pmod 2$, so that $p = \pi \overline{\pi} = x^2 + 3y^2 \equiv 3 \pmod 4$ and therefore $p \equiv 7 \pmod{12}$. We also have $x \equiv -1 \pmod 3$ and $y \equiv -1 \pmod 3$, and we can choose a generator $\alpha$ of $\F_p^\ast$ such that $3y \equiv (2 \alpha^{(p-1)/3} + 1)x \pmod p$. Such a generator exists because $(3y - x)/(2x)$ is a primitive $3$th root of unity in $\F_p^\ast$; indeed, it is different from $1$ because $x \neq y$, and using $y^2/x^2=-1/3$ one checks that it cubes to $1$. Combining this choice of $\alpha$ with $\beta = 3$ into an element $a \in (\Z / 7p \Z)^\ast$ by means of the Chinese remainder theorem, we see from Example~\ref{ex:exampleq7} that the corresponding set $S$ satisfies
\[ | \, S \cap (S + 1) \, | = \frac{1}{36} (5p + 5 + 22x + 12y) \equiv \frac{1}{1}(5(1^2 + 3\cdot 5^2) + 5 + 22\cdot 1 + 12 \cdot 5) \equiv -2 \pmod 7,\]
as wanted.

Because this construction applies to every Eisenstein prime satisfying~\eqref{eq:eisensteindirichlet}, of which there are infinitely many, we obtain the existence of infinitely many primes $p$ with the desired properties.
\end{proof}

\begin{remark} \label{rmk:densityq7}
Note that, here again, there are other congruence classes to which the above reasoning applies besides that of $3 + 10\zeta$ mod $84$. Indeed, we could have worked with any  $z = z_1 + z_2 \zeta$ satisfying $z_1 \equiv 1 \pmod 2$, $z_2 \equiv 2 \pmod 4$, $z_1 + z_2/2 \equiv -1 \pmod 3$, $\gcd(z_1^2 + z_1 z_2 + z_2^2, 84) = 1$ and which is such that the formula for $| \, S \cap (S + 1) \, |$ from~\eqref{eq:formulaScapSp1q7} applied to $x = z_1 + z_2/2$ and $y = z_2/2$ yields a value congruent to $-2$ mod $7$. The reader can check that these properties hold for the following $36$ congruence classes mod $84$: 
\[ \begin{array}{cccccc} 1 + 50\zeta, & 3 + 10\zeta, & 5 + 42\zeta, & 7 + 74\zeta, & 9 + 34\zeta, & 11 + 66\zeta, \\
13 + 14\zeta, & 21 + 82\zeta, & 23 + 30\zeta, & 25 + 62\zeta, & 27 + 22\zeta, & 29 + 54\zeta, \\ 
31 + 2\zeta, & 33 + 46\zeta, & 35 + 78\zeta, & 37 + 26\zeta, & 39 + 70\zeta, & 41 + 18\zeta, \\
43 + 50\zeta, & 45 + 10\zeta, & 47 + 42\zeta, & 49 + 74\zeta, & 51 + 34\zeta, & 53 + 66\zeta, \\
55 + 14\zeta, & 63 + 82\zeta, & 65 + 30\zeta, & 67 + 62\zeta, & 69 + 22\zeta, & 71 + 54\zeta, \\
73 + 2\zeta, & 75 + 46\zeta, & 77 + 78\zeta, & 79 + 26\zeta, & 81 + 70\zeta, & 83 + 18\zeta, \\ 
\end{array} \]
where we note that the latter $18$ are just obtained from the former $18$ by adding $42$. 

Denote by $P_7$ the set of prime numbers $p$ with the requested property, i.e., for which there exists an $a \in (\Z / 7p \Z)^\ast$ meeting the requirements from Theorem~\ref{th:constructionpq} and for which the corresponding set $S$ satisfies $| \, S \cap (S+1) \, | \equiv -2 \pmod 7$. As in Remark~\ref{rmk:densityq5}, one can check that every $p \in P_7$ arises as the norm of an Eisenstein prime $\pi$ belonging to one of the above $36$ congruence classes. Moreover, for an Eisenstein prime $\pi$ in one of these congruence classes, it can be checked that no generator of $(\overline{\pi})$ (i.e., none of the six elements $\pm \overline{\pi}$, $\pm \zeta \overline{\pi}$, $\pm \zeta^2 \overline{\pi}$) belongs to that same congruence class. In other words, the list contains no analogue of the exceptional case $13 + 10\mathbf{i}$ from Remark~\ref{rmk:densityq5}. Mimicking the rest of the reasoning from Remark~\ref{rmk:densityq5}, and using that the ray class group of $\Q(\zeta)$ for modulus $(84)$ contains $432$ elements, we then conclude that $\delta(P_7) = \frac{36}{432}=\frac{1}{12}$.
\end{remark}

\begin{example}
The Eisenstein prime $\pi = 3 + 10\zeta = 8 + 5\mathbf{i}\sqrt{3}$ of norm $p = \pi \overline{\pi} = 139$ of course satisfies~\eqref{eq:eisensteindirichlet}. The generator $\alpha = 2$ of $\F_{139}^\ast$ meets the requirement $3y \equiv (2 \alpha^{(p-1)/3} + 1)x  \pmod p$ for $x=8$ and $y=5$. With $\beta = 3 \in \F_7^\ast$ this combines into $a = 836 \in (\Z / 973\Z)^\ast$. The corresponding set $S=S_{973}(836)$ satisfies $| \, S \cap (S + 1) \, | = \frac{1}{36}(5p + 5 + 22\cdot 8 + 12\cdot 5) = 26 \equiv -2 \pmod 7$. Note that $836^{65}\equiv26\pmod{973}$ and $\gcd(65,138)=1$, so $S_{973}(836)=S_{973}(26)$. This is the value that we find in Table~\ref{tab:constructionqprime}.
\end{example}

\begin{remark}\label{rem:infinitestrictq7}
    Arguing in the same way as in Remark \ref{rem:infinitestrictq5}, we see that the Neumaier graphs arising from the construction in Theorem \ref{th:constructionpq} for $q=7$ are necessarily strictly Neumaier if $p\geq127$. Hence, this construction produces infinitely many strictly Neumaier graphs for $q=7$.
\end{remark}

\subsection{Infinitely many infinite families of Neumaier graphs} \label{ssec:infiniteinf}

Finally, building on Example~\ref{ex:examplen6}, we show that there exist infinitely many $q$'s for which there exist infinitely many prime numbers $p$ admitting an $a \in (\Z / pq \Z)^\ast$ with the requested properties. First we prove a lemma about a specific system of modular equations.

\begin{lemma} \label{lem:conic}
Let $q$ be a product of (not necessarily distinct) prime numbers that are congruent to $1$ modulo $6$.
There exist integers $z_1, z_2$ such that 
\begin{enumerate}[label=(\roman*)]
    \item \label{it:cond1} $z_1 \equiv 1 \pmod 2$, $z_2 \equiv 2 \pmod 4$, 
    \item \label{it:cond2} $z_1 + z_2/2 \equiv -1 \pmod 3$, $z_2/2 \equiv 0 \pmod 3$,
    \item \label{it:cond3} $\gcd(z_1^2 + z_1 z_2 + z_2^2, 12q) = 1$,
    \item \label{it:cond4} $\left(2(z_1^2 + z_1z_2 + z_2^2) + 2 + 4z_1 + 2z_2 \right)/36 \equiv -2 \pmod q$.
\end{enumerate}
\end{lemma}
\begin{proof}
It suffices to find integers $z_1, z_2$ meeting condition~\ref{it:cond4} and $\gcd(z_1^2 + z_1z_2 + z_2^2, q) = 1$, which are conditions modulo $q$. Indeed, such integers can be transformed into integers satisfying conditions~\ref{it:cond1}--\ref{it:cond4} by further imposing $z_1 \equiv 5 \pmod{12}$ and $z_2 \equiv 6 \pmod{12}$, which can be done using the Chinese remainder theorem. When looking for integers $z_1, z_2$ meeting condition~\ref{it:cond4} and $\gcd(z_1^2 + z_1z_2 + z_2^2, q) = 1$, it suffices to assume that $q = \ell^e$ for some prime $\ell \equiv 1 \pmod 6$ and some exponent $e \geq 1$, again by the Chinese remainder theorem. 

If $e = 1$ we are looking for a point $(z_1, z_2) \in \F_\ell^2$ on the conic
\begin{equation} \label{eq:conic} 
Z_1^2 + Z_1Z_2 + Z_2^2 + 2Z_1 + Z_2 + 37 = 0
\end{equation}
which moreover satisfies $z_1^2 + z_1z_2 + z_2^2 \neq 0$ or, equivalently, $2z_1 + z_2 + 37 \neq 0$. One checks that this conic is absolutely irreducible, hence non-singular, and that it has two $\F_\ell$-rational points at infinity, so there are $\ell - 1$ affine points over $\F_\ell$. At most two of these points satisfy the linear equation $2Z_1 + Z_2 + 37 = 0$. Therefore, since $\ell \equiv 1 \pmod 6$ is at least $7$, a point with the desired properties exists. 
 
If $e > 1$ then one again starts from a point $(\overline{z}_1, \overline{z}_2)$ on the conic~\eqref{eq:conic} viewed over $\F_\ell$, making sure that $\overline{z}_1^2 + \overline{z}_1\overline{z}_2 + \overline{z}_2^2 \neq 0$. Since it concerns a non-singular point, at least one of the partial derivatives of the left-hand side of~\eqref{eq:conic} does not vanish at it; let us assume that this is true for $\partial / \partial Z_1$, the other case is completely analogous. Now view the left-hand side of~\eqref{eq:conic} as a polynomial over $\Z / \ell^e \Z$ and substitute an arbitrary lift $z_2$ of $\overline{z}_2$ for $Z_2$. The remaining univariate polynomial in $Z_1$ satisfies the hypotheses of Hensel's lemma~\cite[Thm.\,2.23]{numbertheory} at $\overline{z}_1$, so we can lift the latter to obtain a solution $(z_1, z_2)$ of~\eqref{eq:conic} over $\Z / \ell^e \Z$. The condition $\gcd(z_1^2 + z_1z_2 + z_2^2, q) = 1$ is ensured because $(z_1, z_2)$ reduces to $(\overline{z}_1, \overline{z}_2)$ modulo $\ell$. This concludes the proof of the lemma.
\end{proof}

\begin{theorem}
Let $q$ be a product of (not necessarily distinct) prime numbers that are congruent to $1$ modulo $6$.
There exist infinitely many prime numbers $p$ for which there exists an $a \in (\Z / pq \Z)^\ast$ meeting the requirements from Theorem~\ref{th:constructionpq} and for which $S = S_{pq}(a)$ satisfies $| \, S \cap (S+1) \, | \equiv -2 \pmod q$.
\end{theorem}
\begin{proof}
If $q = 7$ then this follows from Theorem~\ref{thm:infinitely7}, so we can assume $q > 7$ and choose $\beta \in (\Z / q\Z)^\ast$ as in Example~\ref{ex:examplen6}, i.e., such that $\beta^2 = \beta - 1$; such a $\beta$ exists, as was explained there. 
\par Apply Lemma~\ref{lem:conic} to find integers $z_1, z_2$ satisfying~\ref{it:cond1}--\ref{it:cond4}. We then proceed as in Section \ref{ssec:infinite7}: according to Theorem~\ref{thm:dirichlet} there exist infinitely many prime elements $\pi \in \Z[\zeta]$ such that
\begin{equation} \label{eq:eisensteindirichlet2} 
12q \, | \, \pi - (z_1 + z_2 \zeta),
\end{equation}
where we note that $12q$ and $z_1 + z_2 \zeta$ are indeed coprime, thanks to condition \ref{it:cond3} in~Lemma \ref{lem:conic}.
Writing $\pi = c + d\zeta$, this implies that $c \equiv z_1$ and $d \equiv z_2$ modulo $12q$. Note, in view of~\ref{it:cond1}, that $c \neq 0$, $d \neq 0$ and $c \neq -d$. As a consequence $\pi$ cannot be of the form $\pm p, \pm p\zeta, \pm p \zeta^2 = \mp p \pm p\zeta$ for an integer prime $p$. Thus $\pi$ is an Eisenstein prime of the second kind, i.e.~$c^2 + cd + d^2$ is a prime $p \equiv 1 \pmod 3$ (indeed, the case $p = 3$ is easy to rule out).

Define $x = c + d/2$ and $y = d/2$, which are integers because $d \equiv z_2 \pmod{12q}$ is even, again in view of~\ref{it:cond1}. We then have $x \equiv z_1 + z_2 / 2 \pmod {6q}$ and $y \equiv z_2 / 2 \pmod{6q}$. In particular we find that $x \equiv 0 \pmod 2$ and $y \equiv 1 \pmod 2$, again by~\ref{it:cond1}, so that $p = \pi \overline{\pi} = x^2 + 3y^2 \equiv 3 \pmod 4$ and therefore $p \equiv 7 \pmod{12}$. Next, one sees that $x \equiv -1 \pmod 3$ and $y \equiv 0 \pmod 3$ in view of~\ref{it:cond2}. We can find a generator $\alpha$ of $\F_p^\ast$ that satisfies $3y \equiv (2 \alpha^{(p-1)/3} + 1)x \pmod p$, see Section \ref{ssec:infinite7}. Choosing  such a generator and combining it with $\beta$ using the Chinese remainder theorem, we then find an element $a \in (\Z / pq\Z)^\ast$ such that the corresponding set $S$ satisfies $| \, S \cap (S+1) \, | \equiv -2 \pmod q$; indeed, this follows from~\ref{it:cond4} and~\eqref{eq:formulaScapSp1n6}.

This reasoning applies to each of the infinitely many Eisenstein primes $\pi$ satisfying~\eqref{eq:eisensteindirichlet2}, from which the theorem follows. 
\end{proof}

\begin{example}
We choose $q = 13 \cdot 19 = 247$, and we check that $\beta = 69$ satisfies $\beta^2 = \beta - 1$ in $(\Z / q \Z)^\ast$; we can find $\beta = 69$ using the Chinese remainder theorem, having found $4$ and $12$ as solutions of $X^{2}=X-1$ in $\F_{13}^\ast$ and $\F_{19}^\ast$, respectively. When viewed over $\F_{13}$, the conic~\eqref{eq:conic} admits the point $(z_1, z_2) = (0, 1)$ and it satisfies $z_1^2 + z_1z_2 + z_2^2 \neq 0$. Similarly, over $\F_{19}$ we find that the point $(z_1, z_2) = (0, 14)$ has the requested properties. Modulo $q = 13 \cdot 19$, these points combine into $(z_1, z_2) = (0, 14)$. Finally, by further imposing $z_1 \equiv 5 \pmod{12}$ and $z_2 \equiv 6 \pmod{12}$, we find that $(z_1, z_2) = (2717, 1002)$ satisfies conditions~\ref{it:cond1}--\ref{it:cond4} modulo $12q = 12\cdot13\cdot19$. Within the congruence class of $z_1 + z_2 \zeta$ mod $12q$, we find the Eisenstein prime
\[ \pi = c + d\zeta, \qquad \text{where $c = z_1 - 12q$ and $d = z_2$,} \]
of norm $p = \pi \overline{\pi} = c^2 + cd + d^2 = 817519$. The respective values of $x = c + d/2$ and $y = d/2$ are $254$ and $501$. One checks that $\alpha = 15$ is a generator of $\F_p^\ast$ satisfying $3y = (2\alpha^{(p-1)/3} + 1)x$. Together with $\beta = 69$ this combines into $a = 22890547 \in (\Z / pq \Z)^\ast$, and the corresponding set $S=S_{pq}(a)$ can be seen to satisfy $| \, S \cap (S + 1) \, | = 45446 = 184 \cdot 247 - 2$, which is indeed congruent to $-2$ modulo $q$.
\end{example}

\begin{remark}\label{rem:infinitestrictqgeneral}
    Let $q$ be a product of (not necessarily distinct) prime numbers that are congruent to $1$ modulo $6$. Arguing in the same way as in Remarks~\ref{rem:infinitestrictq5} and~\ref{rem:infinitestrictq7}, we see that the Neumaier graph arising from the construction in Theorem \ref{th:constructionpq} for $q$ is necessarily strictly Neumaier if $p\geq18(q-2)+8\sqrt{18(q-2)+16}+31$. Hence, this construction produces infinitely many strictly Neumaier graphs for $q$.
\end{remark}

\section*{Acknowledgements}
Aida Abiad is partially supported by the FWO (Research Foundation Flanders, No 1285921N). Wouter Castryck is supported by the Research Council KU Leuven grant C14/18/067 and by CyberSecurity Research Flanders with reference VR20192203. Jack H.~Koolen is partially supported by the National Natural Science Foundation of China (No. 12071454), Anhui Initiative in Quantum Information Technologies (No. AHY150000) and the National Key R and D Program of China (No. 2020YFA0713100).

\bibliographystyle{abbrv}
\addcontentsline{toc}{chapter}{Bibliography}
\bibliography{bibliography}{}\newpage

\end{document}